\documentclass[11pt]{amsart}
\usepackage[parfill]{parskip}    
\usepackage{fullpage}
\usepackage{graphicx}
\usepackage{amssymb,amsmath,latexsym}
\usepackage{tikz}
\usetikzlibrary{arrows,decorations.markings,decorations.pathreplacing,calc,matrix,intersections}
\tikzset{->-/.style={decoration={markings,mark=at position #1 with {\arrow{>}}},postaction={decorate}}}

\usepackage{epstopdf}
\usepackage{ifpdf}
\usepackage{color}
\usepackage{float}
\usepackage{amscd,amsmath, mathabx}
\usepackage{amssymb,amsmath,latexsym,color,enumerate,tikz}
\usepackage[all]{xypic}       
\usepackage[varg]{pxfonts}
\usepackage{amsmath,amsthm,amsfonts,amssymb,fancyhdr,graphics,relsize,tikz-cd,mathtools,faktor,tikz,changepage}
\usepackage{graphicx}
\usepackage{placeins}

\usepackage{dsfont}
\definecolor{red}{rgb}{1,0,0} 

 \definecolor{darkgreen}{rgb}{0, .7, 0}

 \definecolor{purple}{rgb}{.7, 0, 1}

\tikzset{mynode/.style={draw,circle,fill=black,inner sep=2pt,outer sep=0.5pt}}

\newtheorem{theorem}{Theorem}[section]
\newtheorem*{theorem*}{Theorem}
\newtheorem*{lemma*}{Lemma}
\newtheorem{proposition}[theorem]{Proposition}
\newtheorem{lemma}[theorem]{Lemma}
\newtheorem{corollary}[theorem]{Corollary}

\theoremstyle{definition}

\theoremstyle{remark}

\usepackage{hyperref}
 
\begin{document}
\title{Non self-similar metabelian groups}
\author{Dessislava H. Kochloukova, Melissa de Sousa Luiz}
\address{State University of Campinas (UNICAMP)}
\email{desi@unicamp.br, m262675@dac.unicamp.br}

\maketitle

\begin{abstract} We show   some sufficient conditions for a finitely presented group $G = A \rtimes Q$, with $A$ and $Q$ abelian and $Krull dim (A) = 2$ to be not self-similar. This is in contrast to the case of $Krull dim(A) = 1$ considered in \cite{K-S}.
\end{abstract}

\section{Introduction}

There is a vast literature on self-similar groups.   A self-similar group (or a state closed group) of degree $m$ is a group $G$ admitting a faithful, self-similar action on $m$-ary rooted tree.Among those groups are some well known examples as the Grigorchuk group \cite{G} and the Gupta-Sidki group \cite{Sidki}.  Recently Olivier showed that there exist finitely generated nilpotent groups that are not self-similar \cite{O}. Another example of a finitely generated nilpotent not self-similar group can be found in \cite{PhD}.

Self-similar groups are always residually finite, but the converse does not hold. As shown by Hall in \cite{Hall} all finitely generated metabelian groups are residually finite. In this paper we aim to describe a class of finitely presented metabelian groups that are not self-similar. In \cite{Brazil2} Dantas and Sidki showed that $\mathbb{Z} \wr \mathbb{Z}$ is not transitive self-similar, by showing that if $G = A \wr B$, with $A$ abelian and $B = \mathbb{Z}^n$, is transitive self-similar then $A$ has finite exponent. But it is well known that $\mathbb{Z} \wr \mathbb{Z}$ is  (intransitive) self-similar.
In \cite{K-S} Kochloukova and Sidki showed that for $G = A \rtimes Q$ where $A$ is viewed as $\mathbb{Z} Q$-module via conjugation , if the Krull dimension of $A$ as $\mathbb{Z} 	Q$-module is 1 ( i.e. the Krull dimension of the ring $\mathbb{Z} Q/ ann_{\mathbb{Z} Q} (A)$ is 1) and the centralizer $C_{Q}(A) = \{ q \in Q \ | \ [q, A] = 1 \}$ is trivial then $G$ is transitive self-similar. This result is somewhat surprising and in this paper we show that Krull dimension 1 case is a very specific case and similar behaviour should not be expected in Krull dimension 2 and probably in higher dimension.

 In the case of transitive self-similar groups the link between virtual endomorphism and self-similarity was pioneered by Nekrashevych and Sidki in \cite{N-S}. The version of this result for (intransitive) self-similar groups was developed by Dantas,  Santos and Sidki in \cite{Brazil}. This was used to study some examples of self-similar groups but many of the examples studied before are of the type $A \rtimes Q$, where the action of $Q$ on $A$ is ``close'' to being free i.e. many examples are modelled by wreath products. The situation considered by Kochloukova and Sidki in \cite{K-S} is quite different,  there the authors show examples $G = A \rtimes Q$, with $A$ and $Q$ abelian, $G$ of homological type $FP_m$  but in this case $A$ is always of finite exponent. Since for metabelian groups the homological type $FP_2$ coincides with finite presentability, see \cite{B-S}, we get examples of finitely presented metabelian  self-similar groups. The first known example of  a finitely presented metabelian  self-similar group was given by Bartholdi, Neuhauser, Woess
 in \cite{B-N-W}.  In \cite{S-S} Skipper and Steinberg realised lamplighter groups $A \wr \mathbb{Z}$  with A a finite abelian group as automaton groups via affine transformations of power series rings with coefficients in a finite commutative ring and gave conditions on the power series that  guarantee that the automaton is reversible or bireversible.
 
  The main result of this paper is that some metabelian groups, corresponding to Krull dimension 2, are not self-similar. Our study was inspired by the question of Dantas whether it is possible to find an example of a finitely presented self-similar metabelian group that contains a copy of $\mathbb{Z} \wr \mathbb{Z}$ and a question of Sidki whether it is possible to classify all finitely  generated metabelian self-similar groups. We could not find an example that answers Dantas's question but our study of possible examples lead us to the main result of this paper and we conjecture  that  a finitely presented self-similar metabelian group cannot contain a copy of $\mathbb{Z} \wr \mathbb{Z}$. Though we do not answer Sidki's question  in full we get some general results and the amount of difficulties we reach in this paper points out that probably the question is very difficult to answer in  its full generality.

  Our approach is to study possible virtual endomorphisms of $G = A \rtimes Q$ using results from commutative algebra developed in the study of $\Sigma$-theory. $\Sigma$-theory was pioneered by Bieri and Strebel that used it to classify all finitely presented metabelian groups in \cite{B-S}. The structure
of the first $\Sigma$-invariant introduced in \cite{B-S} was latter linked by Bieri and Groves to the valuation theory from commutative algebra. They   proved that the complement of  $\Sigma$ in the character sphere  is a rationally defined spherical polyhedron, see \cite{B-G2}.  Later $\Sigma$-invariants were developed for general (non-metabelian) groups and they are often referred to as BNSR-invariants. Some recent results on this topic can be found in \cite{DawidK}, \cite
{K-M}, \cite{W-Z}, \cite{Matt}.

   Let $G = A \rtimes Q$ be a group, where $A$ and $Q$ are abelian. We view $A$ as a right $\mathbb{Z} Q$-module via conjugation i.e. the operation $+$ in $A$ is the restriction of the group operation in $G$ to $A$, the $Q$ action is conjugation ( on the right) i.e. the action of $q \in Q$ on $a \in A$ is  $ a \circ q = q^{-1} a q$. If $G$ is finitely generated, then $Q$ is finitely presented and hence $A$ is finitely generated as a $\mathbb{Z} Q$-module.

  For a ring $R$ we denote by $Krulldim(R)$ the Krull dimension of $R$ i.e. the maximal length $k$ of a chain of prime ideals $P_0 < P_1 < \ldots < P_k$ in $R$.

 {\bf Main Theorem} {\it  Let $G = A \rtimes Q$ be a group, where $A$ and $Q$ are abelian,  $Q = \mathbb{Z}^s$, $s \geq 2$.
We view $A$ as a right $\mathbb{Z} Q$-module via conjugation and assume that

1) $A$ is a cyclic  $\mathbb{Z} Q$-module, say $A \simeq \mathbb{Z} Q/ I$,  $A$  is a $\mathbb{Z}$-torsion-free  integral domain and $Krull dim (A) = 2$;

2)  for every prime number $p$ the ring $A/ pA$ is  an infinite  integral domain;

3) the image of a non-trivial element of $Q$ in the field of fractions of $A$ is not algebraic over $\mathbb{Q}$. In particular $C_Q(A) = 1_Q$;

4) $G$ is finitely presented.

Then $G$ is not a self-similar group.}

\medskip
 
  We observe that the condition that $A$ is $\mathbb{Z}$-torsion-free, i.e. $A$ has zero characteristic,  is important, as in \cite{K-S} were constructed examples of $G$ transitive self-similar with $A$ of Krull dimension bigger than 1 but  $A$ is of finite exponent.

  The core of the proof of the Main Theorem  is based on the technical Theorem  \ref{oldA1} that describes possible structural  restrictions on virtual endomorphisms. The proof of Theorem \ref{oldA1}  uses substantially $\Sigma$-theory.
  We call the rings $A$ that satisfy condition 3 from Theorem \ref{oldA1}  homothety rigid rings. In section \ref{HRR} we show that the domain $A$ from the Main Theorem  is a homothety rigid ring and the starting point is an old theorem of Puiseux-Newton  that parametrizes an algebraic curve in the plane using power series.

We first prove our Main Theorem assuming an extra  condition that $G$ is virtual-endomorphism finite, see Theorem \ref{MainThm}.

\medskip 
{\bf Definition} {\it  Let $G  = A \rtimes Q$, with $A$ and $Q$ abelian, $G$ finitely generated.    Consider  a finite set of virtual endomorphisms   $$f^{(i)} : A_i \rtimes Q_i \to G,$$  such that for $1 \leq i \leq k$ we have 
   
   1) $f^{(i)} (A_i) \subseteq A$,  
   
   2) there is NOT a positive integer $m_i$ such that $m_i A \subseteq A_i$ and $f^{(i)}(m_i A) = 0$,
   
   3)  $f^{(i)}({Q_i}) \subseteq Q$, 
   
   4) $f_0^{(i)} = f^{(i)} |_{Q_i}$ is injective.
   
  We say that $G  = A \rtimes Q$ is virtual-endomorphism-finite if for any finite set of virtual endomorphisms as above    we have that $\{f_0^{(i)} \}_{1 \leq i \leq k}$ generates a finite group of injective homomorphisms $\widetilde{Q} \to Q$, where  $\widetilde{Q}$ is a subgroup of finite index in $\cap_{1 \leq i \leq k} \ Q_i$. }
  
  \medskip

  In Section \ref{Proof-main} we prove  that the assumptions of the Main Theorem  imply that  $G$ is virtual-endomorphism finite.
  
   In Section \ref{examples} we consider  a special example $G = A \rtimes Q$  that satisfies the Main Theorem, where  $$A = \mathbb{Z}[x^{\pm 1}, 1/(x+1)], Q = \langle q_1, q_2 \rangle \simeq \mathbb{Z}^2$$ and $Q$ acts on $A$ via conjugation with $q_1$ acting by multiplication with $x$ and $q_2$ acting by multiplication with $x+1$.   The group  $\mathbb{Z} \wr \mathbb{Z} \simeq \mathbb{Z}[x^{\pm 1}] \rtimes \langle q_1 \rangle$ embeds in $G$. The group $G$ was our original failed attempt to embed $\mathbb{Z} \wr \mathbb{Z}$ in a finitely presented self-similar group   and it was the motivation behind the results in this paper.

  As the paper uses substantially methods and ideas from commutative algebra and $\Sigma$-theory
  we include preliminary section on these topics.

{\bf Acknowledgements} The first named author was partially supported by grant CNPq 305457/2021-7 and the second named author was supported by a PhD grant CNPq 141727/2021-7.

\section{Preliminaries on self-similar groups and virtual endomorphisms}

Let $\mathcal{T}_m$ be the $m$-ary tree, that starts with a unique root and every vertex has precisely $m$ descendents. We write $\mathcal{T}_m^{(0)}, \ldots$,$\mathcal{T}_m^{(m-1)}$  for the $m$-ary subtrees of $\mathcal{T}_m$ that start at the vertices in the first layer of $\mathcal{T}_m$. Let $G$ be a group acting on the tree in the way it preserves descendents. For every $g \in G$ we have a decomposition
\begin{equation} \label{decomposition} 
g = (g_0, \ldots, g_{m-1}) \sigma
\end{equation}
where $\sigma$ is a permutation in $S_m$ that describes the action of $g$ on the first layer of the tree $\mathcal{T}_m$ and each $g_i$ acts on $\mathcal{T}_m$ by fixing the root and all vertices outside  $\mathcal{T}_m^{(i)}$. A group $G$ is self-similar if for every $g$ the elements $g_0, \ldots, g_{m-1}$, called states of $g$,  belong to $G$ i.e. $G$ is state closed. We say that $G$ is a transitive self-similar group if it acts transitively on the first layer of $\mathcal{T}_m$.

A virtual endomorphism is a group homomorphism $f : H \to G$, where $H$ is a subgroup of finite index in $G$. It is called simple if there is no non-trivial normal subgroup $K$ of $G$ such that $K \subseteq H$ and $f(K) \subseteq K$.

\begin{theorem} \cite{N-S}
$G$ is a transitive self-similar group if and only if there is a simple virtual endomorphism   $f : H \to G$.
\end{theorem}

This result was recently generalized to intransitive actions.

\begin{theorem} \cite{Brazil}, \cite{PhD}
$G$ is a self-similar group acting with $k$ orbits on $\mathcal{T}_m$ if and only if there are virtual endomorphisms   $f_i : H_i \to G$ for $ 1 \leq i \leq k$ such that there is no non-trivial normal subgroup $K$ of $G$ such that $K \subseteq \cap_{ 1 \leq i \leq k} H_i$ and $f_i(K) \subseteq K$ for $ 1 \leq i \leq k$.
\end{theorem}

The idea behind the virtual endomorphisms is that each one represents one orbit under the action of $G$ on the first level of the tree.

\section{Preliminaries on $\Sigma$-theory and commutative algebra} \label{Sigma}

\subsection{$\Sigma$-theory}
Let $Q$ be a finitely generated abelian group.
For $\chi \in Hom(Q, \mathbb{R}) \setminus \{ 0 \}$ consider the monoid 
$$
Q_{\chi} = \{ q \in Q \ | \ \chi(q) \geq 0 \}.$$
In $Hom(Q, \mathbb{R}) \setminus \{ 0 \}$ there is an equivalence relation $\sim$  given by $\chi_1 \sim \chi_2$ if and only if there is a positive real number $r$ such that $\chi_1 = r \chi_2$. By definition  the character sphere of $Q$ is
$$
S(Q) = Hom (Q, \mathbb{R}) \setminus \{ 0 \} / \sim $$
and $[\chi]$ is the equivalence class of $ \chi \in Hom(Q, \mathbb{R}) \setminus \{ 0 \}$  i.e. $ [\chi] = \mathbb{R}_{>0} \chi$.

Let $A$ be a finitely generated $\mathbb{Z} Q$-module.
The Bieri-Strebel invariant $\Sigma_A(Q)$ was defined in \cite{B-S} as
$$
\Sigma_A(Q) = \{ [\chi] \in S(Q) \ | \ A \hbox{ is finitely generated as } \mathbb{Z} Q_{\chi}-\hbox{module} \}.$$
The classification of finitely presented metabelian groups is described in the following result.

\begin{theorem} \cite{B-S} \label{fin-pres-0} Let $ 1 \to A \to G \to Q \to 1$ be a short exact sequence of groups with $A$ and $Q$ abelian, $G$ finitely generated. Then the following conditions are equivalent:

1) $G$ is finitely presented;

2) $G$ is of homological type $FP_2$;

3) $A$ is 2-tame as $\mathbb{Z} Q$-module, i.e. $S(Q) = \Sigma_A(Q) \cup - \Sigma_A(Q)$.
\end{theorem}

A group $G$ is of type $FP_2$ if the trivial $\mathbb{Z} G$-module $\mathbb{Z}$ has a projective resolution where all projectives in dimension $\leq 2$ are finitely generated. This is equivalent to the relation module of $G$ with respect to a finite generating set being finitely generated as $\mathbb{Z} G$-module, where $G$ acts via conjugation. In general finite presentability implies type $FP_2$ but there are  special groups that are $FP_2$ but are not finitely presented (but they are not metabelian).

Let $R$ be a commutative ring with unity. A valuation $v : R \to \mathbb{R}_{\infty}$ is a map such that 

1) $v(0) = \infty$,

2) $v(ab) = v(a) + v(b)$ for all $a,b \in R$,

3) $v( a + b) \geq {\min} \{ v(a), v(b) \}$ for all $a,b \in R$. 

Note that $v^{-1} (\infty)$ is a prime ideal in $R$ that is not necessarily the zero one.

By definition 
$$\Sigma_A^c(Q) = S(Q) \setminus \Sigma_A(Q).$$ 

\begin{theorem} \cite[Thm. 8.1]{B-G2} \label{thm-valuation1} Let $Q$ be a finitely generated abelian group and $A$ be a finitely generated $\mathbb{Z} Q$-module. Then $[\chi] \in \Sigma_A^c(Q)$ if and only if there is a  valuation $v : \mathbb{Z} Q/ ann_{\mathbb{Z} Q} A \to \mathbb{R}_{\infty}$ such that the restriction of $v$ on the image of $Q$ is induced by $\chi$.
\end{theorem}

Actually \cite[Thm. 8.1]{B-G2} is slightly more general as it treats modules over $R Q$, where $R$ is a commutative  ring with unity  and $v(R) \geq 0$. Note that for $R = \mathbb{Z}$ the condition $v(\mathbb{Z}) \geq 0$ is automatic. We state the result in the form above as we will need it in this form later.

{\bf Example}

Set $Q = \langle x,y \rangle \simeq \mathbb{Z}^2$ and  $A = \mathbb{Z} [x^{ \pm 1}, y^{ \pm 1}]/ ( y - x - 1)$. Let $$w : A \to \mathbb{R}_{\infty}$$ be a valuation and $v = w | _{\mathbb{Z}} : \mathbb{Z} \to \mathbb{R}_{\infty}$ , $\chi = w |_Q$. Then $w ( y - x - 1) = w(0) = \infty$, so there are 3 possibilities:

a) $w(x) = w(y) \leq w(1) = 0$, hence $\chi(x) = \chi(y) \leq {0}$ corresponds to the ray $\{ (\lambda, \lambda) = (\chi(x), \chi(y)) \ | \ \lambda \leq 0 \}$;

b) $w(y) = w(1) = 0 \leq w(x)$, hence $\chi(y) = 0 \leq \chi(x)$ corresponds to the ray $\{ ( \lambda, 0) = (\chi(x), \chi(y))  \ | \ \lambda \geq 0 \}$;

c) $w(x) = w(1)= 0  \leq w(y)$, hence $\chi(x) = 0 \leq \chi(y)$ corresponds to  the ray $\{ (0, \lambda)= (\chi(x), \chi(y)) \ | \ \lambda \geq {0} \}$.

Thus we have 3 rays that start at the point $(0,0)$. Projecting to $S(Q)$ we obtain that
$$\Sigma_A^c(Q) = \{[\chi_0], [\chi_1], [\chi_2]\} $$
where $\chi_0(x) = \chi_0(y) = -1$, $\chi_1(x) = 1, \chi_1(y) = 0$ and $\chi_2(x) = 0, \chi_2(y)  = 1$.

\begin{corollary} \label{biject} 
Let $Q$ be a finitely generated abelian group with a subgroup $\widetilde{Q}$ of finite index, $A$ be a finitely generated $\mathbb{Z} Q$-module and $B$ is a $\mathbb{Z} \widetilde{Q}$-submodule of $A$ such that $[A : B] < \infty$. Then there is a bijection
$$\tau = \tau_{A, B, Q, \widetilde{Q}} : \Sigma_A^c(Q) \to \Sigma_{B}^c(\widetilde{Q})$$
given by restriction i.e. $\tau([\chi]) = [\widetilde{\chi}]$, where $\widetilde{\chi} = \chi |_{\widetilde{Q}}$.
\end{corollary}

\begin{proof} The map $ \tau_{A, B, Q, \widetilde{Q}}$ can be decomposed as the composition map  $\tau_{A, B, \widetilde{Q}, \widetilde{Q}} \circ \tau_{A, A, Q, \widetilde{Q}}$. Since $[A : B] < \infty$ we have that $\tau_{A, B, \widetilde{Q}, \widetilde{Q}}$ is a bijection. Since $[Q : \widetilde{Q}] < \infty$, by Theorem \ref{thm-valuation1}  or by  \cite[Prop. 2.3]{B-S}  we have that $\tau_{A, A, Q, \widetilde{Q}}$ is a bijection.
\end{proof}

\begin{lemma} \label{auxiliar}
Let $Q$ be a finitely generated abelian group with a subgroup $\widetilde{Q}$ of finite index and $A = \mathbb{Z} Q/ I$ be an integral domain. Let $B \not= 0$ be a $\mathbb{Z} \widetilde{Q}$-submodule of $A$. Then
$\Sigma_A^c(\widetilde{Q}) = \Sigma_{B}^c(\widetilde{Q})$.
\end{lemma}

\begin{proof}
We view $A$ as a $\mathbb{Z} \widetilde{Q}$-module via the restriction of the $Q$-action to $\widetilde{Q}$. By \cite[Prop. 2.2]{B-S}
$$ \Sigma_{A}^c(\widetilde{Q})  =  \Sigma_{B}^c(\widetilde{Q}) \cup  \Sigma_{A/ B}^c(\widetilde{Q}), \hbox{ in particular } \Sigma_{B}^c(\widetilde{Q}) \subseteq  \Sigma_{A}^c(\widetilde{Q}).$$ 

We aim to prove that  $ \Sigma_{A}^c(\widetilde{Q}) \subseteq  \Sigma_{B}^c(\widetilde{Q})$.
Let $[\widetilde{\chi}] \in \Sigma_A^c(\widetilde{Q})$ and let $\chi : Q \to \mathbb{R}$ be the homomorphism that is the unique extension of $\widetilde{\chi}$.
Since $\widetilde{Q}$ has finite index in $Q$ we have   by  \cite[Prop. 2.3]{B-S}  that $[\chi] \in \Sigma_A^c(Q)$. Then there is a valuation
$$v : A = \mathbb{Z} Q/ I \to \mathbb{R}_{\infty}$$
whose restriction on the image of $Q$ is induced by $\chi$. Then $v( a q) =  v(a) + \chi(q)$ for $a \in A, q \in Q$
and for $C = \mathbb{Z} \widetilde{Q}/ (I \cap \mathbb{Z} \widetilde{Q})$
$$v |_{C} :  C \to \mathbb{R}_{\infty}$$
is a valuation such that $v(c q) = v(c) + \widetilde{\chi} (q)$ for $c \in C, q \in \widetilde{Q}$. Then  by Theorem \ref{thm-valuation1} $[\widetilde{\chi}] \in \Sigma_C^c(\widetilde{Q})$, hence  $\Sigma_A^c(\widetilde{Q}) \subseteq \Sigma_C^c(\widetilde{Q})$ .

Let $b \in B \setminus \{ 0 \}$. Then $B_0 = b \mathbb{Z} \widetilde{Q} \subseteq B$ and so $\Sigma_{B_0}^c(\widetilde{Q}) \subseteq \Sigma_B^c(\widetilde{Q})$. Note that 
$B_0 \simeq \mathbb{Z} \widetilde{Q}/ (I \cap  \mathbb{Z} \widetilde{Q}) = C$ sending $b q$ to the image of $q$  in $C$ for $q \in \widetilde{Q}$. Here we used that $A$ is an integral domain. Thus 
$$\Sigma_A^c(\widetilde{Q}) \subseteq \Sigma_C^c(\widetilde{Q}) = \Sigma_{B_0}^c(\widetilde{Q}) \subseteq \Sigma_B^c(\widetilde{Q}).  $$
\end{proof}

We recall some definitions and results. Let $S$ be a fixed subgroup of $\mathbb{R}$ with respect to the operation +.
 We call $C \subseteq \mathbb{R}^s$ a {\it convex} polyhedron if
 $$C = H_1 \cap H_2 \cap \ldots \cap H_r,$$
 where
 $$H_i = \{ (x_1, \ldots, x_s) \in \mathbb{R}^s  \ | \   \sum_j q_{i,j} {x_j }\geq a_i \},$$
  where $s$ is the torsion-free rank of $Q$. $H_i$ is {\it rationally defined over $S$} if each $q_{i,j} \in \mathbb{Q}$ and $a_i \in S$.
  The dimension $dim C$ is the dimension of the affine space spanned by $C$.
  
  A polyhedron (rationally defined over $S$)  is
  $$\Delta = C_1 \cup \ldots \cup C_n$$
  where each $C_i$ is convex polyhedron ({rationally} defined {over} $S$). We say that $\Delta$ is homogeneous of dimension $m$ if each $C_i$ has dimension $m$.
  
  Let $R$ be a commutative ring with unity and $v : R \to \mathbb{R}_{\infty}$ be a valuation. Let $Q$ be a finitely generated abelian group and $A$ an algebra over the group algebra $RQ$ given by a ring homomorphism $\kappa : RQ \to A$. Then
  $\Delta^v_A(Q) \subseteq Q^* = Hom(Q, \mathbb{R})$ is given by
  $$\Delta^v_A(Q) = \{ \chi : Q \to \mathbb{R} \ | \ \hbox{ there is a valuation } w : A \to \mathbb{R}_{\infty}, w \circ \kappa |_R = v,  w \circ \kappa |_Q = \chi \}
  $$
  
  \begin{theorem} \cite[Thm. 5.2]{B-G2} \label{val-val1}
  Let $A$ be a domain, $k \subseteq A$ a field endowed with a valuation $v : k \to \mathbb{R}_{\infty}$ and $Q$ a finitely generated  subgroup of the unit group $U(A)$ of $A$. Then $\Delta^v_A(Q) \leq Q^* = Hom(Q, \mathbb{R})$ is a homogeneous polyhedron of dimension that equals the transcendence degree of $k(Q)$ over $k$ and rationally defined over $v(k^{\times}) \subseteq \mathbb{R}$.
  \end{theorem}

\begin{theorem}  \cite[Thm. 5.4]{B-G2} \label{Dedekind} 
Let $R$ be a Dedekind domain, $Q$ a finitely generated abelian group, $A$ a Noetherian $R Q$-algebra. Then there exists a finite set of prime ideals $\Pi$ of $R$ such that for all $P \in Spec (R) \setminus \Pi$ we have
$$\Delta^{v_P}_{A}(Q) = \Delta^0_{A}(Q)$$
where $v_P$ is the $P$-adic valuation of $R$.
\end{theorem}

The following is a particular case of  \cite[Thm. A a)]{B-G}.

\begin{theorem} \cite[Thm. A a)]{B-G} \label{rigidity1} Let $k \subseteq K$ be an extension of fields, $Q$ a finitely generated multiplicative subgroup of $K^{\times} =  K \setminus \{ 0 \}$. Suppose no non-trivial element of $Q$ is algebraic over $k$. Then $Q^* = Hom (Q, \mathbb{R})$ {is spanned} as $\mathbb{R}$-vector {space} by $\Delta^0_K(Q)$.
\end{theorem}

\subsection{Krull dimension of commutative rings}

We recall that   for a ring $R$ the Krull dimension of $R$ denoted by  $Krulldim(R)$ is the maximal length $k$ of a chain of prime ideals $P_0 < P_1 < \ldots < P_k$ in $R$. 

Examples: $Krulldim(\mathbb{Z}) = 1$, if $R$ is a field then $Krulldim(R) = 0$.

\begin{theorem} \cite{W} Suppose $R$ is a commutative Noetherian ring with $1$ and $R[x_1, \ldots, x_n]$ is the polynomial ring on commuting variables. Then $Krulldim(R[x_1, \ldots, x_n]) = n + Krulldim(R)$. 
\end{theorem}

As a corollary we have that $Krulldim(\mathbb{Z}[x_1, \ldots, x_n]) = n + 1$. After localization we have for the ring  of  Laurent polynomials $\mathbb{Z}[x_1^{\pm 1} , \ldots, x_n^{\pm 1}]$ that $Krulldim(\mathbb{Z}[x_1^{\pm 1} , \ldots, x_n^{\pm 1}]) = n+1.$

The following result is a corollary of the Lying over theorem from commutative algebra, see \cite[Thm~5.9]{H}, \cite[5.10]{A-M}.

\begin{theorem}  \label{int} Suppose $A \subseteq B$ is an integral extension of commutative  rings. Then $Krulldim(A) = Krulldim(B)$.
\end{theorem}

\begin{theorem} \label{Noether}  (Noether Normalization)
Let $k$ be a field and $A$ be a finitely generated commutative $k$-algebra. Then there exist elements $y_1, \ldots, y_d \in A$ that are algebraically independent over $k$ and $A$ is a finitely generated module over the polynomial ring $S = k[y_1, \ldots, y_d]$.
\end{theorem}

{\it Remark} Note that in the above theorem since $A$ is integral over $S$ by Theorem \ref{int}  $d = Krulldim(S) = Krulldim(A)$. If furthermore $A$ is an integral domain then by Theorem \ref{Noether} $d$ is the transcendence degree of $K$ over $k$, where $K$ is the field of fractions of $A$.

\section{The main technical result}

The following result  is well known but for completeness we give a proof.

\begin{lemma} \label{prime}
Let $R$ be a commutative ring with $1$ and $J_1, \ldots , J_k$ prime ideals of $R$, $I$ an ideal of $R$ such that $I \subseteq \cup_{1 \leq i \leq k} J_i$. Then there is $i_0$ such that $I \subseteq J_{i_0}$.
\end{lemma}

\begin{proof} Without loss of generality $k \geq 2$.  Assume the result is wrong and  consider a   counter example with minimal $k$,  then $I \not\subseteq \cup_{1 \leq i \not= j \leq k} J_i$ for every $ 1 \leq j \leq k$. Thus  for every $j$  there is $x_j \in I \setminus \cup{_{1 \leq i \not= j \leq k}} J_i \subseteq J_j$. Set $y_j = x_1 \ldots \widehat{x_j} \ldots x_k \in  (I \cap ( \cap_{ 1 \leq i \not= j \leq k} J_i)) \setminus J_j$. Then $\sum_{1 \leq j \leq k} y_j \in I \setminus (\cup_{1 \leq i \leq k} J_i)$, a contradiction.
\end{proof}

\begin{lemma} \label{minus1}  Let $Q$ be a finitely generated abelian group and  $A \simeq \mathbb{Z} Q/ I$ be a domain of characteristic 0. Then  $B = A \otimes_{\mathbb{Z}} \mathbb{Q}$ has Krull dimension at most $Krulldim(A) - 1$.
 \end{lemma}
 
 \begin{proof}
 Note that $B \simeq A S^{-1}$ where $S = \mathbb{Z} \setminus \{ 0 \}$.
 Let $0 = V_0 \subsetneq V_1  \subsetneq \ldots  \subsetneq V_d$ be a maximal chain of prime ideals in $B$ i.e. $Krulldim(B) = d$. Note that each prime ideal in $B$ is a localization of a prime ideal in $A$ i.e.  each $V_i = P_i S^{-1}$, where  $P_i = V_i \cap A$ and  $ P_0 = 0 \subsetneq P_1 \subsetneq \ldots \subsetneq P_d $ is a chain of  prime ideals in $A$, hence $Krulldim (A) \geq d$. 
 
 Suppose that $Krulldim(A) = d$, then $P_d$ is a maximal ideal of $A$. Note that by the main result of \cite{Jim} for any polycyclic group   $H$,  every maximal ideal in the group algebra  $\mathbb{Z} H$  is of finite index.  In particular for $H = Q$ we obtain that  $P_d$ has finite index in $A$. Hence $A/ P_d$ has finite exponent, so $(A/P_d)S^{-1} = 0$, so $B = A S^{-1} = P_d S^{-1}$, a contradiction with $P_d S^{-1}$ is a maximal ideal in $B$. Thus $Krulldim(A) > d = Krulldim(B)$. 
 \end{proof}
The following is a technical result that will be used later 
to deduce strong restrictions on virtual endomorphisms of a metabelian group $G = A \rtimes Q$.

\begin{theorem} \label{oldA1} Let $s \geq 2$ be an integer, $Q \simeq \mathbb{Z}^s$ be a finitely generated abelian group and $Q_0$ a subgroup of finite index in $Q$. Let $A$ be a $\mathbb{Z} Q$-module and $A_0$ be a subgroup of finite index in $A$ that is a $\mathbb{Z} Q_0$-submodule. Let $f : A_0 \to A$ be a homomorphism of abelian groups such that $$f(a_0 q_0) = f(a_0) f_0(q_0)\hbox{ for }a_0 \in A_0, q_0 \in Q_0,$$ where $f_0: Q_0 \to Q$ is an injective homomorphism of groups.

Assume furthermore that

1) $A = \mathbb{Z} Q/ I$  is a $\mathbb{Z}$-torsion-free integral domain  of  $Krull dim (A) = 2$;

2)  for every prime number $p$ the ring $A/ pA$ is   infinite domain;

3)  If $Q = \langle x_1, \ldots, x_s \rangle$ and if there are positive integers  $\widetilde{n}, \widetilde{c}_1, \ldots, \widetilde{c}_s$ and a ring homomorphism $$\mathbb{Z} \langle x_1^{\widetilde{n}}, \ldots, x_s^{\widetilde{n}} \rangle / \mathbb{Z} \langle x_1^{\widetilde{n}}, \ldots, x_s^{\widetilde{n}} \rangle \cap I \to \mathbb{Z} \langle x_1^{\widetilde{c}_1}, \ldots, x_s^{\widetilde{c}_s} \rangle / \mathbb{Z} \langle x_1^{\widetilde{c}_1}, \ldots, x_s^{\widetilde{c}_s} \rangle \cap I $$ that sends $x_i^{\widetilde{n}}$ to $x_i^{\widetilde{c}_i}$ for $ 1 \leq i \leq s$ then $\widetilde{n} = \widetilde{c}_1 = \ldots = \widetilde{c}_s$; 

4) The image of a non-trivial element of $Q$ in the field of fractions of $A$ is never algebraic over $\mathbb{Q}$;

5) $\Sigma_{A}^c(Q)$ does not contain an one dimensional circle $S^1$, obtained by intersecting a two dimensional subspace of $Hom(Q, \mathbb{R}) \simeq \mathbb{R}^{s}$ with the unit sphere $S^{s-1}$.

 Then there is a $\mathbb{Z} Q$-submodule $M = m A + \mu A$ of $A_0$ such that $m$ is a positive integer   and $\mu \in A$ and furthermore  

a) $M$ is abelian subgroup of $A_0$ of finite index;

and either 

b1) $f(M) = 0_A = 1_G$

or

b2) there is   
 a positive integer $k$ and some finite index subgroup $Q_1$ of $Q_0$ such that  $f_0^k : Q_1 \to Q$  is the inclusion map
and  $f(mA) \subseteq mA$, $f(\mu A) \subseteq \mu A$.

\end{theorem}

\medskip
Proof.
As the proof is long we split it in several steps.

{\bf  Step 1.} By construction $[A : A_0] < \infty$ and for $T$ a coset of $Q_0$ in $Q$ we have that $\cap_{t \in T}  A_0 t$ is a $\mathbb{Z} Q$-submodule ( via conjugation) of finite index in $A$. Hence substituting $A_0$ with this intersection we can assume that $A_0$ is a $\mathbb{Z} Q$-submodule of $A$.

Suppose that $f(A_0) \not= 0$  otherwise b1) holds automatically.

 Note that $f(a_0 q_0) = f(a_0) f_0(q_0)$ for $a_0 \in A_0, q_0 \in Q_0$ implies that $f(A_0)$ is a $\mathbb{Z} \widetilde{Q}$-submodule of $A$,  where $\widetilde{Q} = f_0(Q_0)$.

\medskip
{\bf Claim} There is a bijection
$$ \psi :  \Sigma^c_{f(A_0)}(\widetilde{Q}) \to \Sigma_{A_0}^c(Q_0)$$
 given by $\psi([\widetilde{\chi}]) = [ \widetilde{\chi} \circ f_0]$.
 
 \medskip
 {\bf Proof of the Claim.}
 Since $A = \mathbb{Z} Q/ I$ is a domain   and $f(A_0) \not= 0$,  we have that $ann_{\mathbb{Z} \widetilde{Q}} f(A_0) = \mathbb{Z} \widetilde{Q} \cap I$. Similarly since $A_0 \not= 0$ we have that $ann_{\mathbb{Z} Q_0} A_0 = \mathbb{Z} Q_0 \cap I$. Then by \cite[(1.3)]{B-S2} 
 \begin{equation}  \label{igual1}  \Sigma^c_{f(A_0)}(\widetilde{Q}) =  \Sigma^c_{\mathbb{Z} \widetilde{Q} / (\mathbb{Z} \widetilde{Q} \cap I) }(\widetilde{Q}) \end{equation}  and
 \begin{equation} \label{igual2}  \Sigma_{A_0}^c(Q_0) =  \Sigma_{\mathbb{Z} Q_0/ (\mathbb{Z} Q_0 \cap I)}^c(Q_0) \end{equation}
 Let $\widetilde{f}_0$ be the $\mathbb{Z}$-linear map $\mathbb{Z} Q_0 \to \mathbb{Z} \widetilde{Q}$ induced by $f_0$. Thus $\widetilde{f}_0$ is an isomorphism of rings.
 
 The condition  $f(a_0 q_0) = f(a_0) f_0(q_0)$ for $a_0 \in A_0, q_0 \in Q_0$ implies that for $\lambda \in \mathbb{Z} Q_0$ we have that 
 $f(a_0 \lambda) = f(a_0) \widetilde{f}_0(\lambda)$. Thus fixing one $a_0$ such that $f(a_0) \not= 0$ we obtain that 
 $$0 = f(0) = f(a_0 ( \mathbb{Z} Q_0 \cap I))= f(a_0)  \widetilde{f}_0 ( \mathbb{Z} Q_0 \cap I)$$
 and since $A$ is a domain we conclude that $ \widetilde{f}_0 ( \mathbb{Z} Q_0 \cap I) \subseteq \mathbb{Z} \widetilde{Q} \cap I$.
 Thus we  have an isomorphism of integral domains induced by $\widetilde{f}_0$
 $$\nu : M_1 = \mathbb{Z} Q_0/ (\mathbb{Z} Q_0 \cap I) \to M_2 = \mathbb{Z} \widetilde{Q}/ \widetilde{f}_0(\mathbb{Z} Q_0 \cap I)$$
 and $M_3 = \mathbb{Z} \widetilde{Q} / (\mathbb{Z} \widetilde{Q} \cap I)$ is a quotient of $M_2$. 
 
 Finally since $Q_0$ and $\widetilde{Q}$ are subgroups of finite index in $Q$ we have that $M_1 \leq A$ and $M_3 \leq A$ are integral ring extension i.e. $A$ is finitely generated as $M_1$-module and as $M_3$-module. Then $$Krulldim(M_1) = Krulldim(A) = Krulldim(M_3) = 2.$$ Since $M_1 \simeq M_2$ we have $Krulldim(M_1) = Krulldim(M_2)$, hence 
 $$Krulldim(M_2) = Krulldim(M_3)$$ Finally since $A$ is an integral domain, both $M_1$ and $M_3$ are integral domains, hence $M_2 \simeq M_1$ is an integral domain. Sumarising $M_3$ is a quotient of $M_2$, both are integral domains of the same Krull dimension, hence $M_3 = M_2$. Thus the isomorphism of rings $\nu$ can be rewritten as
 $$\nu : M_1 = \mathbb{Z} Q_0/ (\mathbb{Z} Q_0 \cap I) \to M_3 = \mathbb{Z} \widetilde{Q}/ (\mathbb{Z} \widetilde{Q} \cap I)$$
 Thus there is a bijection
 $$\Sigma_{M_3}^c(\widetilde{Q}) \to \Sigma_{M_1}^c(Q_0)$$ that sends
  $[\widetilde{\chi}]$ to $[ \widetilde{\chi} \circ f_0]$.
 This together with (\ref{igual1}) and (\ref{igual2}) completes the proof of the Claim. 
 
\medskip
Consider the bijection
$$ \gamma  = \tau_{A,A, Q, \widetilde{Q}} \circ \tau_{A, A_0, Q, Q_0}^{-1} : \Sigma_{A_0}^c(Q_0) \to \Sigma^c_A(Q) \to \Sigma_A^c(\widetilde{Q}) = \Sigma^c_{f(A_0)}(\widetilde{Q})$$
where $\widetilde{Q} = f_0(Q_0)$ and the last equality is given by Lemma \ref{auxiliar}. 
 
Then 
 $$f_{*} = \psi \circ  \gamma : \Sigma_{A_0}^c(Q_0) \to \Sigma_{A_0}^c(Q_0)$$ is a bijective map.  
 
 Set $D= M_1$, where $M_1$ was defined in the proof of the Claim.  Then (\ref{igual2}) can be stated as
 $$\Sigma_{A_0}^c(Q_0) = \Sigma_D^c(Q_0)$$
 By  Theorem \ref{thm-valuation1} $$ \Sigma_{D}^c(Q_0) = \cup_j [\Delta^{v_{j}}_{D}(Q_0)  \setminus \{ 0 \} ]$$ where  $v_{j}$ is a real valuation of the image $R$ of $\mathbb{Z}$ in $D$. We   identify $ R$ with $ \mathbb{Z}$. By Theorem \ref{Dedekind} we can take the union to be finite.

 There are two cases to consider.
 
 First, if $v_j^{-1} (\infty) = 0$. In this case the valuation $v_j$ can be extended to a valuation $w_j$ of $\mathbb{Q}$. Then $$\Delta^{v_{j}}_{D}(Q_0) = \Delta^{w_{j}}_{D \otimes_{\mathbb{Z}} \mathbb{Q}} (Q_0)$$
 
 By Lemma \ref{minus1} we have that $$Krulldim(D \otimes_{\mathbb{Z}} \mathbb{Q}) \leq 2-1 = 1.$$ Then the transcendence degree of the field of fractions of $D$ as a field extension of $\mathbb{Q}$ is at most 1. Then by Theorem \ref{val-val1} $\Delta^{w_{j}}_{D \otimes_{\mathbb{Z}} \mathbb{Q}}(Q_0)$ is a homogeneous polyhedron of dimension at most 1 i.e. finite union of segments,rays and lines.

Second, if  $v_j^{-1} (\infty) = p \mathbb{Z}$, where $p$ is a prime number. Then $$\Delta^{v_{j}}_{D}(Q_0) = \Delta^{u_{j}}_{D/ pD} (Q_0),$$ where $u_j$ is the valuation of $\mathbb{Z}/ p \mathbb{Z}$ induced by $v_j$, i.e., the zero one. As $D$ is an integral domain we have that $$Krulldim(D/pD) \leq Krulldim(D) - 1  =1.$$ Note that $D_0 = D/ pD$ is a Noetherian ring, hence has finitely many minimal prime ideals $P_1, \ldots, P_l$. For any valuation $w: D_0 \to \mathbb{R}_{\infty}$ we have that $w^{-1} (\infty)$ is a prime ideal of $D_0$, which contains some  $P_i$.  We conclude that $w$ induces a valuation $D_i : = D_0/ P_i \to 
\mathbb{R}_{\infty}$. Thus $$ \Delta^{u_{j}}_{D/ pD} (Q_0) = \cup_{1 \leq i \leq l}  \Delta^{u_{j}}_{D_i} (Q_0).$$ Note that $Krulldim(D_i) \leq Krulldim(D_0) \leq 1$. Then transcendence degree of the field of fractions of $D_i$ as a field extension of the field with $p$ elements is at most 1. Then by Theorem \ref{val-val1} $ \Delta^{u_{j}}_{D_i} (Q_0)$ is a homogeneous polyhedron of dimension at most 1, i.e., a finite union of segments, rays and lines. In the case of dimension 0 the homogeneous polyhedron is a finite set of isolated points that corresponds to segments of length 0. 

We conclude that in all cases $\Delta^{v_{j}}_{D}(Q_0)$ is a finite union  of segments, rays and lines. Furthermore it suffices to consider normalised valuations $v_j$ i.e. we can substitute $v_j$ by $\lambda v_j$ for $ \lambda$ a positive real number, thus we can assume that  $Im (v_j) \simeq \mathbb{Z}_{\infty}$ or $Im (v_j) = \{ 0, \infty \}$ and thus by Theorem \ref{val-val1} the above segments, rays and lines are rationally defined over $\mathbb{Z}$ i.e. are rationally defined.

Now we consider a segment or a ray or a line that is contained in $\Delta^{v_{j}}_{D}(Q_0) \subseteq Hom(Q_0, \mathbb{R})\simeq \mathbb{R}^s$. This segment or ray or a line together with the origin span a  vector space of dimension  at most 2 that intersects the unit sphere $S^{s-1}$ in $\mathbb{R}^s$ in circle $S^1$ or in two antipodal points. We identify $S(Q_0)$ with $S^{s-1}$. Thus when we project  the segment or the ray or the line excluding the origin to $S(Q_0)$ we obtain an arc in $S^1$  or a point.  Thus $ \Sigma_{D}^c(Q_0)$ is a finite union of arcs  and isolated points, and the union of some arcs cannot give a whole circle $S^1$ as described above since $ \Sigma_{D}^c(Q_0)$ does not contain such circle $S^1$ by condition 5) of the statement (note that since by Corollary \ref{biject} $\tau_{A, A_0, Q, Q_0}$ is a bijection condition 5) is equivalent to  $\Sigma_{A_0}^c(Q_0)$ does not contain an one dimensional circle $S^1$, obtained by intersecting a two dimensional subspace of $Hom(Q_0, \mathbb{R}) \simeq \mathbb{R}^{s}$ with the unit sphere $S^{s-1}$). Since $ \Sigma_{D}^c(Q_0)$  is a closed subset of $S(Q_0)$ the arcs that appear in the above union are closed, actually some of the original arcs can be open from one or both sides but after the union we get a new decomposition of  $ \Sigma_{D}^c(Q_0)$  as a finite union of closed arcs and isolated points, i.e. we have a spherical polyhedron that is rationally defined since the corresponding $\Delta$'s are rationally defined.

 We define the {\bf boundary} points $\partial  \Sigma_{D}^c(Q_0)$ as the end points of the closed arcs and the isolated points. Then we conclude that $\partial  \Sigma_{D}^c(Q_0) \not= \emptyset$ and furthermore $\partial  \Sigma_{D}^c(Q_0)$ are discrete points in $S(  Q_0)$.

 Thus $\partial  \Sigma_{A_0}^c(Q_0) $ is a finite non-empty set, say  with $d$ elements.
 Note that $f_*$ permutes the elements of  $\partial  \Sigma_{A_0}^c(Q_0) $. 
 Then $$f_{*}^k |_{\partial  \Sigma_{A_0}^c(Q_0)} = id |_{\partial  \Sigma_{A_0}^c(Q_0)} \hbox{ for }k = d!.$$

Note that condition 4 together with Theorem \ref{rigidity1} applied for $\Delta_{K}^0(Q_0)$, where $K$ is the field of fractions of $D$,  imply that for every element $[\chi] \in S(Q_0)$ we have that the character $\chi$ is a $\mathbb{R}$-linear combination of characters $\chi_1, \ldots, \chi_j$ for some $$\{[\chi_1], \ldots, [\chi_j] \} \subseteq [ \Delta_{K}^0(Q_0)]  \subseteq  [ \Delta_{D}^0(Q_0)] \subseteq  \Sigma_{D}^c(Q_0) = \Sigma_{A_0}^c(Q_0).$$ Hence $\chi$ is a $\mathbb{R}$-linear combination of characters  from $ \partial  \Sigma_{A_0}^c(Q_0)$. Then since $f_{*}^k |_{\partial  \Sigma_{A_0}^c(Q_0)}$ is the identity map  we deduce that $Q$ has a basis $x_1, \ldots, x_s$ such that for some positive integers $n_1, \ldots, n_s, c_1, \ldots, c_s$  we have that $f_0^k$ is defined in $\langle x_1^{n_1}, \ldots, x_s^{n_s} \rangle$,
$$\langle x_1^{n_1}, \ldots, x_s^{n_s} \rangle \leq Q_0 \hbox{ and }f_0^k(x_i^{n_i} ) = x_i^{c_i}.$$
By substituting $n_1, \ldots, n_s$ with their least common multiple  we can assume that $n_1 = \ldots = n_s = n$.

{\bf  Step 2. }  Let $E$ be the image  of $\mathbb{Z} Q^n$ in $A$,  where $n$ was defined above. We claim that there  are  $\mu \in E$  and an integer  $m> 1$ such that $$[A : m A + \mu A] < \infty \hbox{ and }m A + \mu A \subseteq A_0.$$
Since $[A : A_0] < \infty$ we can choose  an integer $m> 1$ such that $m A \subseteq A_0$.
We decompose $m = p_1^{z_1} \ldots p_u^{z_u}$, where $2 \leq p_1 < p_2 < \ldots < p_u$ are primes.   We prove the existance of $\mu$ in the following Claim.

\medskip
{\bf Claim.} There is $\mu \in \cap_{1 \leq i \leq u} ((A_0 \cap E)  \setminus p_i A) = (A_0 \cap E) \setminus  \cup_{1 \leq i \leq u} p_i A.$

\medskip
 {\bf Proof of Claim.} Let $B = A_0 \cap E$.  We need to show that $B  \not\subseteq \cup_{1 \leq i \leq u} p_i A$. Suppose the contrary i.e. $B  \subseteq \cup_{1 \leq i \leq u} p_i A$.  Note that by the assumptions of  Theorem \ref{oldA1} $A/ p_i A$ is a domain, so each $p_i A$ is a prime ideal of $A$. Since $B$ is an additive subgroup of $A$, $BA$ is an ideal of $A$ contained in $ \cup_{1 \leq i \leq u} p_i A$. Then by Lemma \ref{prime} there is some $i_0$ such that  $B \subseteq BA \subseteq p_{i_0} A$. Since $\mathbb{Z} Q^n \subseteq \mathbb{Z} Q$ is an integral extension of rings, we have that $A$ is finitely generated as $E$-module,  i.e., there are some $a_1, \ldots, a_t \in A$ such that
 $$ A = \sum_{1 \leq i \leq t} a_i E.$$ Since $A_0$ is of finite index in $A$ we deduce that $B$ has finite index in $E$,   i.e. the additive group $E/ B$ is finite, hence finitely generated. Thus there are some elements $ e_1, \ldots, e_r \in E$ such that 
 $$E = B + \sum_{1 \leq j \leq r} \mathbb{Z} e_j.$$
 Hence
 $$ A = \sum_{1 \leq i \leq t} a_i E =  \sum_{1 \leq i \leq t} a_i (B + \sum_{1 \leq j \leq r} \mathbb{Z} e_j) =  \sum_{1 \leq i \leq t} a_iB +  \sum_{1 \leq i \leq t, 1 \leq j \leq r} \mathbb{Z} a_i e_j \subseteq p_{i_0} A + \sum_{1 \leq i \leq t, 1 \leq j \leq r} \mathbb{Z} a_i e_j$$
Hence $A/ p_{i_0} A$ is a finitely generated abelian group of exponent $p_{i_0}$, so finite, a contradiction with condition 2) from the assumptions that  $A/ p A$ is infinite for every prime $p$. This completes the proof of the Claim.

\medskip

Set $$I_i = p_i A + \mu A$$ an ideal of $A$ for each $ 1 \leq i \leq u$. Recall that $A$ is $\mathbb{Z}$-torsion free, so $p_iA \not= 0$.  Since $A/ p_iA$ is a an  integral domain and $p_i A \not=0$ we have that $Krulldim(A/ p_i A) < Krulldim(A) = 2$. Then using that $A/ p_iA$ is infinite domain  we can deduce that $$Krulldim(A/ p_i A) = 1.$$  Thus every non zero ideal of  $A/ p_i A$ has finite index, in particular
 $$[A : I_i] < \infty.$$
We can prove by induction on the positive integer $z$ that $A / I_i^z$ is always finite. Indeed this holds for $z = 1$, suppose it holds for $z$. Then, since $A$ is Noetherian, $I_i^z/ I_i^{z+1}$ is a finitely generated ideal of $A/ I_i^{z+1}$. Thus we can view $I_i^z/ I_i^{z+1}$ as a $A/ I_i^{z+1}$-module, but the action of $I_i/ I_i^{z+1}$ is trivial. Thus $I_i^z/ I_i^{z+1}$ is a finitely generated $A/ I_i$-module. Since $A/ I_i$ is finite, we conclude that $I_i^z/ I_i^{z+1}$ is finite.

Since $p_i$ and $p_j$ are coprime integers for $i \not= j$, $p_i A$ and $p_jA$ are coprime ideals i.e. their sum is $A$. Hence $ I_i$ and $I_j$ are coprime ideals in $A$ and thus $I_i^{z_i}$ and $I_j^{z_j}$ are coprime ideals of $A$. Then
$$I_1^{z_1} \ldots I_u^{z_u} = I_1^{z_1} \cap \ldots \cap I_u^{z_u} \hbox{ has finite index in }A.$$
Note that $I_i^{z_i} \subseteq p_i^{z_i} A + \mu A$, hence
$$I_1^{z_1} \ldots I_u^{z_u} \subseteq p_1^{z_1} \ldots p_u^{z_u} A + \mu A = m A + \mu A$$
This shows that  $[A : m A + \mu A] < \infty$.

{\bf Step 3.}   Let $Q_1$ be a subgroup of finite index in $Q_0$ and $T$ be a transversal of $Q_1$ in $Q$.  
 Denote by overlining the image via the canonical projection $\mathbb{Z} Q \to \mathbb{Z} Q/ I = A$ i.e. for $t \in T$, $\overline{t}$ is the image of $t$.  
 Recall  that  $\widetilde{f}_0$ is the isomorphism of rings $\mathbb{Z} Q_0 \to \mathbb{Z} \widetilde{Q} = \mathbb{Z} f_0(Q_0)$ induced by $f_0$ that was defined in step 1. Since $A$ is a domain, $f(aq) = f(a) f_0(q)$ for $a \in A_0, q \in Q_0$ and $f(A_0) \not= 0$ we conclude that 
 $\widetilde{f}_0$ induces a map $\widehat{f}_0: \overline{\mathbb{Z} Q_0} \to \overline{\mathbb{Z} \widetilde{Q}}$.

Consider $t \in T$  and recall that the positive integer $m$ was defined in  step 2  under the only condition that $m A \subseteq A_0$. Then $$f |_{m \overline{t {\mathbb{Z} Q_1}}} : m \overline{t \mathbb{Z} Q_1}  \to  f(m \overline{t}) \overline{\mathbb{Z} Q_2}$$ sends $m \overline{t q_1}$ to $f(m \overline{t}) \overline{f_0(q_1)}$, where $Q_2 = f_0(Q_1)$.

If $f(m \overline{t}) = 0$ for every $ t \in T$ then $f(m \overline{t \mathbb{Z} Q_1}) = 0$  for every $ t \in T$ , hence $$f(m A) = 0.$$ 
 Then $$m f( \overline{t} \mu) = f( m \overline{t} \mu) = f(m \overline{t}) \widehat{f}_0( \mu) =0,$$ hence using that $A$ is  commutative with 0 characteristic $$f( \mu \overline{t}) = f( \overline{t} \mu) = 0$$   and so
$$f(\mu A) = f(\mu 
 \overline{\mathbb{Z} Q}  ) =  f( \sum_{t \in T}  \mu   \overline{{t} \mathbb{Z} Q_1}) = \sum_{t \in T} f(\mu \overline{t})  \widehat{f}_0 (\overline{ \mathbb{Z} Q_1}) = 0,$$
i.e. b1) holds. Note that in this case it was sufficient to assume that $f(mA) = 0$ to deduce b1).

Let $Q_2 =   f_0(Q_1) $.
If $f(m \overline{t}) \not= 0$ for some $t \in T$ 
then using that $A$ is a domain  there is an isomorphism $f(m \overline{t}) {\mathbb{Z} Q_2} \simeq {\mathbb{Z} Q_2} / ann_{\mathbb{Z} Q_2}f(m \overline{t}) = {\mathbb{Z} Q_2} / (\mathbb{Z} Q_2 \cap I)$. Similarly  $m \overline{t} \mathbb{Z} Q_1 \simeq {\mathbb{Z} Q_1} / ann_{\mathbb{Z} Q_1} (m \overline{t}) = \mathbb{Z} Q_1 / (I \cap \mathbb{Z} Q_1)$.
Thus the map $$f |_{m \overline{t \mathbb{Z} Q_1}} : m \overline{t \mathbb{Z} Q_1} \to f(m \overline{t}) \overline{\mathbb{Z} Q_2}$$ induces an isomorphism of rings
$$\rho_{Q_1}:  \mathbb{Z} Q_1 / (I \cap \mathbb{Z} Q_1) \to {\mathbb{Z} Q_2} / (I \cap \mathbb{Z} Q_2)$$
and since for the original map $f$ we have that $f(m \overline{t q_1})  = f(m \overline{t}) \overline{f_0(q_1)}$ we have that $$\rho_{Q_1} \hbox{ is induced by }f_0 |_{Q_1}.$$ Note that to define $\rho_{Q_1}$ we needed only that $Q_1$ is a subgroup of finite index in $Q_0$ and that $f(m A) \not= 0$.

Assume from now on that b1) does not hold.
Let $n_0 $ be a positive integer divisable by $n$ such that $f_0(Q^{n_0}) \subseteq Q^n$, where $n$ is defined at the last line of step 1. Recall that $f_0^k(x_i^n) = x_i^{c_i}$ where both $n$ and $c_i$ are positive and $f_0^k$ is well-defined on $Q^n$.
Then consider the groups
$$Q_1 = Q^{n_0^{k+1}}, Q_2 = f_0(Q_1), \ldots, Q_{k+1} = f_0(Q_k)$$
and the corresponding isomorphisms 
$\rho_{Q_1}, \ldots, \rho_{Q_k}$. These isomorphisms are well defined since $f(m A) \not= 0$ and $Q_1, \ldots , Q_{k+1}$ are subgroups of finite index in $Q_0$.
Then the composition
$$\theta : = \rho_{Q_k} \rho_{Q_{k-1}} \ldots \rho_{Q_1} : \mathbb{Z} Q_1 / (I \cap \mathbb{Z} Q_1) \to {\mathbb{Z} Q_{k+1}} / (I \cap \mathbb{Z} Q_{k+1})$$
is an isomorphism induced by $f_0^k |_{Q_1} : Q_1 \to Q_{k+1}$. Then by condition 3 from the assumptions applied for  $\widetilde{n} = n_0^{k+1}$ and $\widetilde{c}_i = n_0^{k+1} c_i/ n$   we have that $ \widetilde{n} = \widetilde{c}_1 = \ldots = \widetilde{c}_s$. Thus $n = c_1 = \ldots = c_s$ and $$f_0^k |_{Q^n} \hbox{ is the inclusion of }Q^n \hbox{ in }Q \hbox{ i.e. }  f_0^k (q) = q \hbox{ for }q \in Q^n.$$

{\bf  Step 4. }  Here we impose extra condition on $\mu$.

Recall that the isomorphism $f_0 : Q_0 \to f_0(Q_0) = \widetilde{Q}$ extends to a ring isomorphism  $\widetilde{f}_0 : \mathbb{Z} Q_0 \to \mathbb{Z} \widetilde{Q}$.

Let $q \in Q^{n^{k}}$. 
Since $f_0^k = id$ we have that $$\{ f_0^i (q) \ | \ i \geq 0 \} = \{ q = q_1, \ldots, q_j \} \hbox{ is a finite set   with }j \leq k.$$  Consider the polynomial $f_q = (X- q_1) \ldots (X- q_j)$, it  has coefficients in the ring $C =(\mathbb{Z} Q^n)^{\widetilde{f}_0}$ of fixed points under  $ \widetilde{f}_0$. Let $R$ be the subring of $\mathbb{Z} Q$ generated by $\mathbb{Z} Q^{n^{k}}$ and $(\mathbb{Z} Q^{n})^{\widetilde{f}_0}$. By the above every element of $R$ is integral over   $(\mathbb{Z} Q^{n})^{\widetilde{f}_0}$. Note that $\mathbb{Z} Q$ is an integral extension of $ \mathbb{Z} Q^{n^{k}}$, hence every element of  $\mathbb{Z} Q$ is integral over $C$. Since $Q$ is finitely generated this implies that $\mathbb{Z} Q \hbox{ is a finitely generated }C\hbox{-module.}$  Hence $A$ is finitely generated as $\overline{C}$-module, where $\overline{C}$ is the image of $C$ in $A$. Note that for $E$ from step 2 we have that $\overline{C} \subseteq E$.

 We claim that $\mu$ can be chosen from $\overline{C}$. For this we need a modified version of the Claim of step 2, namely that $$ \cap_{1 \leq i \leq k} ((A_0 \cap  \overline{C})  \setminus p_i A) = (A_0 \cap \overline{C}) \setminus  \cup_{1 \leq i \leq k} p_i A  \  \not= \emptyset. $$ 
 For this we can repeat the argument from the beggining of step 2 by substituting $E$ with $\overline{C}$ and use that $A$ is finitely generated as $\overline{C}$-module.

{\bf Step 5. }  
Let $T$ be a transversal of $Q^n$ in $Q$.
 Then $$m A + \mu A = \sum_{t \in T} m \overline{t \mathbb{Z} Q^n} +  \sum_{t \in T} \mu \overline{t \mathbb{Z} Q^n}.$$ Then since $\mu \in \overline{C}$ we have
$$m f( \overline{t} \mu) = f( m \overline{t} \mu)  = f(m \overline{t}) \mu,$$
where the last equality follows from the fact that $\mu \in \overline{C}$. Then $m $ divides $ f(m \overline{t}) \mu$ in $A$.

 Suppose that for some $t = t_1$, 
 $m$ does not divide $f(m \overline{t}_1) $ in $A$. Then $m = m_1 m_2$, where $m_1, m_2$ are positive integers, $m_1$ divides $f(m \overline{t}_1)$ in $A$ and $m_1$ is maximal with this property. Then $m_2 > 1$ and 
let $p$ be a prime integral divisor of $m_2$. Note that $p$ depends on $m$ and $\overline{t}_1$ but not on $\mu$.
 Note that $$m f( \overline{t}_1 \mu) =  f( m \overline{t}_1 \mu) = f(m \overline{t}_1) \mu.$$
 Then  using that $A$ has zero characteristic we obtain
$$ \frac{f(m \overline{t}_1)}{m_1} \mu =m_2 f( \overline{t}_1 \mu) \in p A .$$ By the maximality of $m_1$ we have that $\frac{f(m \overline{t})}{m_1} \notin pA$  and since $A/ pA$ is a domain we conclude that $\mu \in pA$. Then $$M = m A + \mu A \subseteq pA + \mu A \subseteq pA,$$ gives a contradiction since $M$ has finite index in $A$ and $pA$ has infinite index in $A$.

 The contradiction we reached implies
that $m$ divides $ f(m \overline{t})$ in $A$ for every $t \in T$  and so there is $a_t \in A$ such that $m a_t = f(m \overline{t})$. Hence
$m f(  \overline{t} \mu) =  f(m \overline{t}) \mu = m a_t \mu$ and since $A$ is $\mathbb{Z}$-torsion-free we conclude that  $f( \overline{t} \mu) = a_t \mu$. Then 

$$f(m A) =  f(  \sum_{t \in T} m \overline{t \mathbb{Z} Q^n}) =  \sum_{t \in T}  f(m \overline{t}) \overline{\mathbb{Z} f_0(Q^n)}  {\subseteq}  \sum_{t \in T} m a_t  \overline{\mathbb{Z} Q} \subseteq m A
$$
and 
$$
f(\mu A) = f(\sum_{t \in T} \mu  \overline{t \mathbb{Z} Q^n}) = \sum_{t \in T}f(\mu  \overline{t}) \overline{ \mathbb{Z} f_0(Q^n)} {\subseteq} \sum_{t \in T} \mu a_t \overline{\mathbb{Z} Q} \subseteq \mu A.$$

Hence
$$f(M) = f(mA + \mu A) \subseteq mA + \mu A = M.$$
Thus condition b2) holds.

In both cases, b1) and b2), $f(mA) \subseteq m A$. 
 Note that  if $m_0$ is an integer divisable by $m$ then $f(m_0A) = f( \frac{m_0}{m} m A) = \frac{m_0}{m} f(m A) \subseteq \frac{m_0}{m} (m A) = m_0 A$.

\section{Homothety rigid rings} \label{HRR}

{\bf Definition} Let $Q$ be a finitely generated free abelian group and $A = \mathbb{Z} Q/ I$
 a domain of zero characteristic. We say that $A$ is $Q$-homothety rigid if condition 3) from  Theorem \ref{oldA1} holds. For simplicity we write $n$ for $\widetilde{n}$ and $c_i$ for $\widetilde{c}_i$.

\begin{theorem} \label{rigid} Let $Q \simeq \mathbb{Z}^s$ be a finitely generated free abelian group, $s \geq 2$ and $A = \mathbb{Z} Q/ I$
 a domain of zero characteristic. Suppose that $A$ has Krull dimension 2 and that 
 no element of $\overline{Q} \setminus \{ 1 \} \subseteq A$ is algebraic over $\mathbb{Q}$, where $\overline{Q}$ denotes the image of $Q$ in $A$. Then $A$ is $Q$-homothety rigid. 
\end{theorem}
\begin{proof}  
1) We reduce first to the case when $Q$ has rank 2. 

Indeed if  there are $n,c_1, \ldots, c_s \in \mathbb{Z}_{>0}$ such that for $Q = \langle x_1, \ldots, x_s \rangle$, $Q_1 = \langle x_1^{c_1}, \ldots, {x_s^{c_s}} \rangle$ there is a well defined ring homomorphism  $$\theta: \mathbb{Z} Q^n/ (I \cap \mathbb{Z} Q^n) \to \mathbb{Z} Q_1/ (I \cap \mathbb{Z} Q_1)$$ that sends $x_i^n$ to $x_i^{c_i}$ for every $1 \leq i \leq s$, then for a subgroup $Q_2 = \langle x_i, x_j \rangle$ of $Q$, for $ 1 \leq i < j \leq s$, we can define $A_2 = \mathbb{Z} Q_2/ (I \cap \mathbb{Z} Q_2)$ and consider the restriction of $\theta$ that gives a ring homomorphism  $$\mathbb{Z} Q_2^n/ (I \cap \mathbb{Z} Q_2^n) \to \mathbb{Z} Q_3/ (I \cap \mathbb{Z} Q_3)$$ where $Q_3 = \langle x_i^{c_i}, x_j^{c_j} \rangle$.

Since $A$ {is} $\mathbb{Z}$-torsion-free and no non-trivial element of  $\overline{Q}_2$  is algebraic over $\mathbb{Q}$, we conclude that $A_2$ has Krull dimension  at least 2. 
Indeed if $Krulldim(A_2) \leq 1$ then by Lemma \ref{minus1} $Krulldim (A_2 \otimes_{\mathbb{Z}} \mathbb{Q}) \leq Krulldim(A_2)  - 1$, hence $Krulldim (A_2 \otimes_{\mathbb{Z}} \mathbb{Q}) = 0$ and by the remark after Theorem \ref{Noether} the transcendence degree  $trdeg_{\mathbb{Q}} K_2 = 0 $ where $K_2$ is the field of fractions of  $A_2 \otimes_{\mathbb{Z}} \mathbb{Q} $ i.e. $K_2$ and hence  $A_2 \otimes_{\mathbb{Z}} \mathbb{Q}$ are finite dimensional over $\mathbb{Q}$, a contradiction with the assumption that no non-trivial element of $\overline{Q}$ is algebraic over $\mathbb{Q}$.

We aim to show that $Krulldim(A_2) = 2$. Suppose that $Krulldim(A_2) > 2$.
Since $A_2$ is an integral domain quotient of $\mathbb{Z} Q_2$ and $Krulldim(\mathbb{Z} Q_2) = 3$, we get $A_2 = \mathbb{Z} Q_2$, hence $2 = trdeg_{\mathbb{Q}} K_2$. Thus for $K$ the field of fractions of $A$ we have that $K_2 \subseteq K$, so 
$$2 = trdeg_{\mathbb{Q}} K_2 \leq   trdeg_{\mathbb{Q}} K = Krulldim (A \otimes_{\mathbb{Z}} \mathbb{Q}) < Krulldim(A) = 2$$
a contradiction. Hence $Krulldim(A_2) = 2$, as claimed.

2) For the case $s = 2$ we apply Proposition \ref{s2}.

\end{proof}

\begin{proposition} \label{s2}
Let $Q = \langle x,y \rangle \simeq \mathbb{Z}^2$ and $R = \mathbb{Z} Q/ J$ {be} a domain of zero characteristic  of  Krull dimension 2 such that the image  of each non-trivial element of $Q$ in $R$ is not algebraic over $\mathbb{Q}$. Let ${n}, c_1, c_2$ be positive integers, $Q_1$ be the subgroup of $Q$ generated by {$x^{n}, y^{n}$}, $Q_2$ the subgroup of $Q$ generated by $x^{c_1}, y^{c_2}$, $R_i$ the subring of $R$ generated by $Q_i$ for $i = 1,2$. Suppose there is a ring homomorphism 
$$\varphi : R_1 \to R_2$$ that sends {$x^{n}$ to $x^{c_1}$ and $y^{n}$ to $y^{c_2}$}. Then  $n = c_1 = c_2$.
\end{proposition}

\begin{proof}
Both $R_1$ and $R_2$ are domains. Since $R$ is finitely generated over $R_i$ then $R$ and $R_i$ have the same Krull dimension 2  for $i = 1,2$. By construction $\varphi$ is surjective. Then $R_1/ Ker ( \varphi) \simeq R_2$ is a domain of  Krull dimension 2. If $Ker(\varphi) \not= 0$ then using that $0$ is a prime ideal of $R_1$ we conclude that the Krull dimension of $R_1$ is bigger than the Krull dimension of $R_1/ Ker ( \varphi)$ , a contradiction. Thus $\varphi$ is a ring isomorphism.

Let $S_i = R_i \otimes_{\mathbb{Z}} \mathbb{Q} = R_i {(\mathbb{Z} \setminus 0 )^{-1}}$ a domain  for $i = 1,2$, $S =  R \otimes_{\mathbb{Z}} \mathbb{Q} = R {(\mathbb{Z} \setminus 0 )^{-1}}$ that are all $ \mathbb{Q}$-algebras of Krull dimension 1 (if the Krull  dimension of $S$ is 2 then $x,y$ are  algebraically independent elements over $\mathbb{Q}$, hence over $\mathbb{Z}$, hence $R \simeq \mathbb{Z} [X,Y]$  is a polynomial ring on two variables, hence it is of Krull dimension 3, a contradiction).

It is easy to see that 
there is $F \in \mathbb{Q}[X,Y] \subseteq \mathbb{C} [X,Y]$ such that $F(x,y) = 0$ in $S$.
 Indeed, we can consider the field of fractions of $S$ and consider it as a finite extension of $\mathbb{Q}(x)$. Then take $f$ the minimal polynomial of $y \in S$ over $\mathbb{Q} (x)$. Then by multiplying $f$ with some non-zero polynomial from $\mathbb{Q}[X]$ we can get $F$. This should be done in such a way that $F$ does not have a factor in $\mathbb{Q}[X]$.
Note that $F \in \mathbb{Q}[X][Y]$ can be viewed as a polynomial with variable $Y$ that  is irreducible over $\mathbb{Q}[X]$.

 By Puiseux-Newton Theorem there is a positive integer $d$ and a power series
 $$ g = g(t)  \in \mathbb{C}[[t]]$$ such that
 $$F(t^d, g(t) ) = 0.$$
 Note that since $F(0,0)$ is not {necessarily} 0, the constant part of $g$ is not {necessarily} 0. 
 This gives  an {embedding} $$\tau : ff(S) \to \mathbb{C}((t))$$ that sends $x$ to $t^d$ and $y$ to $g(t)$, where $ff(S)$ denotes the field of fractions of $S$ and $ \mathbb{C}((t))$ is the field of fractions of $\mathbb{C}[[t]]$.  By assumption there is an isomorphism $\varphi : R_1 \to R_2$ that induces an isomorphism $$\psi:
 ff ( \tau(R_1)) = \mathbb{Q}({t^{dn}, g(t)^{n}}) \to ff(\tau(R_2)) = \mathbb{Q}(t^{d c_1}, g(t)^{c_2})$$ that is identity on $\mathbb{Q}$ and sends {$t^{dn}$} to $ t^{dc_1}  $,$g(t)^{n}$ to $g(t)^{c_2}$, where as before $ff$ denotes field of fractions. 
 
 Another way of interpreting the Puiseux-Newton Theorem is that the algebraic closure of $\mathbb{C} ((t))$ is $K = \cup_{m > 0 } \mathbb{C} ((t^{1/m}))$. Let $\epsilon_m$ be a primitive $m$-th root of unity.
 Set $$L = \mathbb{Q} ( \{ \epsilon_m, t^{1/m}, g^{1/m} | m > 0  \}) \subseteq {K}$$ and $L_1 = ff ( \tau(R_1))$. Since $L $ is an algebraic field extension of $L_1$ and $K$ is algebraically closed the map $\psi$ extends to a homomorphism of fields
 $$\widetilde{\psi} : L \to K$$
 Since {$\psi (t^{dn}) =  t^{dc_1}$} we can choose $\widetilde{\psi}$ such that
 $$\widetilde{\psi} (t^{1/m}) = t^{c_1/ {n} m},$$
 $$\widetilde{\psi} (g^{1/m}) =  g^{c_2/ {n} m}  u_m$$ and
 $$\widetilde{\psi} ( \epsilon_m) = \epsilon_m^{z_m} \hbox{ for some } z_m \in \{ 1,2, \ldots , m-1 \},$$
 where $u_m$ are appropriate roots of 1 in $\mathbb{C}$.
 By construction $Im (\widetilde{\psi}) \subseteq L$, thus we can itterate $\widetilde{\psi}$ as many times as we want. Then
 $$0 = \widetilde{\psi}^j (0) = \widetilde{\psi}^j( F(t^d, g(t))) = 
   F(\widetilde{\psi}^j (t^d), \widetilde{\psi}^j(g(t)))=
   $$
   $$ F( t^{d (c_1/ n)^j}, g(t)^{ (c_2/ n)^j} v_j)$$ where $v_j$ is a root of 1 in $\mathbb{C}$. Set $t_j = t^{(c_1/ n)^j}$, hence $t = t_j^{(n/ c_1)^j}$.
   Hence we have
   $$F(t_j^d, g( t_j^{(n/ c_1)^j})^{ (c_2/ n)^j} v_j) = 0$$
   Since there is an obvious isomorphism of $K$ that sends $t_j$ to $t$ and is identity on $\mathbb{C}$ we have that
   $$F(t^d, g( t^{(n/ c_1)^j})^{ (c_2/ n)^j} v_j) = 0$$
   Note that $F(t^d, Y) = 0$ is a non-zero polynomial
with variable $Y$, with 
$$\{   g( t^{(n/ c_1)^j})^{ (c_2/ n)^j} v_j | j \geq 0\}$$  a subset of the set of roots. Since a polynomial has only finitely many roots in a fixed field, in our case $L$, we conclude that there is an infinite set of positive integers
$$j_1 < j_2 < j_3 < \ldots $$
   such that
  $$  g( t^{(n/ c_1)^j})^{ (c_2/ n)^j} v_j$$ have the same value for all $j \in \{ j_1, j_2, \ldots \}$.
  
  Now we consider the  power series $$g(t) = a_0 + a_i t^i + \hbox{ higher terms }$$
  where $a_i \not= 0$.
  
 Note that we have 
  \begin{equation}  \label{eq} g( t^{(n/ c_1)^{j_1}} )^{ (c_2/ n)^{j_1}} v_{j_1} =
  g( t^{( n/ c_1)^{j_2}} )^{ (c_2/ n)^{j_2}} v_{j_2}\end{equation}
  Hence
  $$ ( a_0 + a_i t^{i(n/ c_1)^{j_1}}  + \hbox{ higher terms } )^{ (c_2/ n)^{j_1}} v_{j_1} = ( a_0 + a_i t^{i (n/ c_1)^{j_2}}  + \hbox{ higher terms } )^{ (c_2/ n)^{j_2}} v_{j_2}. $$

  {1$^{st}$} case. $g = \lambda t^z$ for some $ z > 0$, $\lambda \in \mathbb{C} \setminus \{ 0 \}$.
By (\ref{eq})
  $$ \lambda^{ (c_2/ {n})^{j_1}} t^{z(n c_2/ c_1 {n})^{j_1}} v_{j_1} = \lambda^{ (c_2/ n)^{j_2}}  t^{z(n c_2/ c_1 n)^{j_2}} v_{j_2}.$$
  Comparing the exponents of $t$  and using that $j_1 \not= j_2$ we get $c_1 = c_2$.
 Suppose $c_2 \not= n$.  Then  since $\lambda^{ (c_2/ n)^{j_1}} v_{j_1} = \lambda^{ (c_2/ n)^{j_2}} v_{j_2}$ we conclude that $\lambda$ is algebraic over $\mathbb{Q}$.
  
   Since  $x = t^d, y = \lambda t^z$ imply $q = x^{-z} y^d = \lambda^d$ and $q$ belongs to the image of $Q$ in $\mathbb{C}((t))$. This contradicts the fact that no non-trivial element of  $\overline{Q}$ is algebraic over $\mathbb{Q}$. Hence $c_2 = n$.

    {2$^{nd}$} case. Suppose that $a_0 \not= 0$ and  $g \not= \lambda t^z$ for  $ z \geq 0$.
  Thus
$$  ( a_0 + a_i t^{i({n}/ c_1)^{j_1}}  + \hbox{ higher terms } )^{ (1/ {n})^{j_1}} = b_0 + b_i   t^{i({n}/ c_1)^{j_1}} + \hbox{ higher terms }$$
 where $b_0^{n^{j_1}} = a_0$, $b_i \not= 0$, $b_0 \not= 0$. Similarly 
$$  ( a_0 + a_i t^{i({n}/ c_1)^{j_1}}  + \hbox{ higher terms } )^{ (1/ {n})^{j_2}} = \beta_0 + \beta_i   t^{i({n}/ c_1)^{j_2}} + \hbox{ higher terms }$$
 where $\beta_0^{n^{j_2}} = a_0$, $\beta_i \not= 0$, $\beta_0 \not= 0$.
Then
$$ (b_0 + b_i   t^{i({n}/ c_1)^{j_1}} + \hbox{ higher terms } )^{c_2^{j_1}} v_{j_1} = ( \beta_0 + \beta_i   t^{i({n}/ c_1)^{j_2}} + \hbox{ higher terms })^{c_2^{j_2}} v_{j_2}$$
and so
$$
(b_0^{c_2^{j_1}} + c_2^{j_1} b_0^{c_2^{j_1}-1} b_i t^{i({n}/ c_1)^{j_1}} + \hbox{ higher terms } ) v_{j_1} =
( \beta_0^{c_2^{j_2}} + c_2^{j_2} \beta_0^{c_2^{j_2}-1} \beta_i   t^{i({n}/ c_1)^{j_2}} + \hbox{ higher terms }) v_{j_2}.
$$
Thus $t^{i({n}/ c_1)^{j_1}} =  t^{i({n}/ c_1)^{j_2}}$, hence $i({n}/ c_1)^{j_1} = i({n}/ c_1)^{j_2}$  and this combined with $j_1 \not= j_2$ implies   ${n} = c_1$.
The problem is symmetric with respect to exchanging $x$ by $y$, thus similarly we have ${n} = c_2$.

{3$^{rd}$} case. Suppose that $a_0 = 0$ and   $g \not= \lambda t^z$ for  $ z \geq 0$. Then we have for some $a_i \not= 0, a_k \not= 0$
$$g = a_i t^i + a_k t^k + \hbox{ higher terms} = t^i ( a_i + a_k t^{k-i}+ \hbox{ higher terms} )$$
 Then
  $$ ( a_i t^{i({n}/ c_1)^{j_1}}  + a_k t^{k({n}/ c_1)^{j_1}}  
  {+}  \hbox{ higher terms } )^{ (c_2/ {n})^{j_1}} v_{j_1} =$$ $$ (a_i t^{i({n}/ c_1)^{j_2}}  +  a_k t^{k({n}/ c_1)^{j_2}}  
  {+}\hbox{ higher terms } )^{ (c_2/ {n})^{j_2}} v_{j_2}. $$
  Thus
  $$  t^{i({n} c_2/ c_1 {n})^{j_1}}  ( a_i  + a_k t^{(k-i)({n}/ c_1)^{j_1}}   +  \hbox{ higher terms } )^{ (c_2/ {n})^{j_1}} v_{j_1}= $$
  $$ t^{i({n} c_2/ c_1 {n})^{j_2}} ( a_i  + a_k t^{(k-i)({n}/ c_1)^{j_2}}   +  \hbox{ higher terms } )^{ (c_2/ {n})^{j_2}} v_{j_2}.$$
Then comparing the lowest powers of $t$ we get ${i({n} c_2/ c_1 {n})^{j_1}}  = {i({n} c_2/ c_1 {n})^{j_2}}$, hence ${c_2/ c_1}= 1$, which implies that $c_1 = c_2$.
  
  Note that
$$  ( a_i  + a_k t^{(k-i)({n}/ c_1)^{j_1}}   + \hbox{ higher terms } )^{ (1/ {n})^{j_1}} = b_0 + b_i   t^{ (k-i)({n}/ c_1)^{j_1}} + \hbox{ higher terms },$$
 where $b_0^{n^{j_1}} =  a_i \not=0$, $b_i \not= 0$. Similarly 
$$  (  a_i  + a_k t^{(k-i)({n}/ c_1)^{j_1}}  + \hbox{ higher terms } )^{ (1/ {n})^{j_2}} = \beta_0 + \beta_i   t^{(k-i)({n}/ c_1)^{j_2}} + \hbox{ higher terms },$$
 where $\beta_0^{n^{j_2}} =  a_i  \not= 0$, $\beta_i \not= 0$.
Then
$$ (b_0 + b_i   t^{(k-i)({n}/ c_1)^{j_1}} + \hbox{ higher terms } )^{c_2^{j_1}} v_{j_1} = $$ $$  ( \beta_0 + \beta_i   t^{(k-i)({n}/ c_1)^{j_2}} + \hbox{ higher terms })^{c_2^{j_2}} v_{j_2}$$
and so
$$
 (b_0^{c_2^{j_1}} + c_2^{j_1} b_0^{c_2^{j_1}-1} b_i t^{(k-i)({n}/ c_1)^{j_1}} + \hbox{ higher terms } ) v_{j_1} =$$
 $$ 
( \beta_0^{c_2^{j_2}} + c_2^{j_2} \beta_0^{c_2^{j_2}-1} \beta_i   t^{(k-i)({n}/ c_1)^{j_2}} + \hbox{ higher terms }) v_{j_2}.
$$
Comparing the exponents of $t$ that appear in the above equality we conclude that
$$ (k-i)({n}/ c_1)^{j_1} =  (k-i)({n}/ c_1)^{j_2}
$$
This combined with $j_1 \not= j_2$ implies 
 ${n} = c_1$. 

\end{proof}

\section{More on virtual endomorphisms of metabelian groups} \label{virtual}

\begin{lemma} \label{L1} Let  $G = A \rtimes Q$ be a finitely generated group, where $A$ and $Q$ are abelian. Let $f : H \to G$ be a virtual endomorphism such that $H = A_0 \rtimes Q_0$, $f(A_0) \subseteq A$. Then there is a virtual endomorphism $\widehat{f} : H \to G$ such that $f$ and $\widehat{f}$ coincide on $A_0$ and $\widehat{f}(Q_0) \subseteq Q$.
\end{lemma}

\begin{proof} We write $a^q$ for $ q^{-1} a q$, where $a \in A, q \in Q$. For $q \in Q_0$ we have $f(q) = a_q \widetilde{q}$ for some $a_q \in A, \widetilde{q} \in Q$. Then we define $\widehat{f}(q) = \widetilde{q} \hbox{ and }  \widehat{f} |_{A_0} = f| _{A_0}.$
Note that for $a \in A_0, q \in Q_0$ we have
$$\widehat{f}(a^q) = f(a^q) = f(a)^{f(q)} = f(a)^{a_q \widetilde{q}} = f(a)^{\widetilde{q}} = \widehat{f}(a)^{\widehat{f}(q)}$$
Since for $q_1, q_2 \in Q_0$ we have $a_{q_1 q_2} \widetilde{q_1 q_2} = f(q_1 q_2) = f(q_1) f(q_2) = a_{q_1} \widetilde{q}_1 a_{q_2} \widetilde{q}_2 = a_{q_1} ( a_{q_2}^{\widetilde{q}_1^{-1}}) \widetilde{q}_1 \widetilde{q}_2$ we obtain that
$$\widehat{f}(q_1 q_2) = \widetilde{q_1 q_2} = \widetilde{q_1} \widetilde{q_2} = \widehat{f}(q_1) \widehat{f}(q_2).$$

\end{proof}

Let $G = A \rtimes Q$ be a  group with $A$ and $Q$ abelian groups. Recall that we view $A$ as a right $\mathbb{Z} Q$-module, where $Q$ acts via conjugation and the operation  $+$ in $A$ is the underlying group operation restricted to $A$. We denote by $\circ$ the action of $Q$ on $A$ i.e. $$\hbox{for }a \in A, q \in Q \hbox{ we have }a \circ q = q^{-1} a q.$$ 
The neutral element of $A$ considered as a $\mathbb{Z} Q$-module is denoted by $0_A$ and it coincides with the neutral element $1_G$ of $G$ considered as a group with multiplicative notation. Note that we use multiplicative notation for the group operation in $Q$. If $G$ is finitely generated as a group, then $A$ is finitely generated as a $\mathbb{Z} Q$-module. If $A$ is a cyclic $\mathbb{Z} Q$-module, then $A \simeq \mathbb{Z} Q/ I$, where $I = ann_{\mathbb{Z} Q}(A)$ is the annihilator of $A$ in $\mathbb{Z} Q$. In this case $A$ has the additional structure of a ring via the isomorphism $A \simeq \mathbb{Z} Q/ I$ and this ring is an integral domain precisely when $I$ is a prime ideal.

Recall that a virtual endomorphism $f : H \to G$ is a group homomorphism, where $H$ is a subgroup of finite index in $G$.

\begin{lemma} \label{L2} Let $G = A \rtimes Q$ be a finitely generated group, where $A$ and $Q$ are abelian. We view $A$ as a $\mathbb{Z} Q$-module via conjugation. Suppose that $A \simeq \mathbb{Z} Q/ I$ is a cyclic $\mathbb{Z} Q$-module and a domain such that $C_Q(A) = 1_Q$. Let $f : H \to G$ be a virtual endomorphism such that $H = A_0 \rtimes Q_0$ and $f(A_0) \not\subseteq A$. Then there is a $\mathbb{Z} Q$-submodule $C$ of $A$ such that $A/ C$ is a finitely generated abelian group and $f(C) = 1_G$.
\end{lemma}

\begin{proof}
Note that since $H$ has finite index in $G$, we have that $A_0$ has finite index in $A$ and $Q_0$ has finite index in $Q$.
Let $$A_1 = \{ a \in A_0 \ | \ f(a) \in A \}.$$  Recall that for $ a \in A, q \in Q$ we write $a \circ q$ for $ q^{-1} a q$.
Note that for $a \in A_1, q \in Q_0$ we have $f(a \circ q) = f(a^q) = f(a) ^{f(q)} \in A^{f(q)} = A$ i.e. $A_1$ is a $\mathbb{Z} Q_0$-submodule of $A_0$.

 Let $a_0 \in A_0 \setminus A_1$, hence $f(a_0) \notin A$. Since $[ A_1, a_0] = 1_G$ we have $[f(A_1), f( a_0)] = 1_G$. Then $f(A_1) \circ (\pi( f(a_0)) - 1) = 0_A$, where $\pi : G \to Q$ is the canonical projection. Since $ a_0 \notin A_1$ we have that $\pi(f(a_0)) \in  Q \setminus 1_Q$. Since $A$ is a domain and $C_Q(A) = 1_Q$
we conclude that $f(A_1) = 0_A =  1_G$.

Set $T$ a transversal of $Q_0$ in $Q$ such that $ 1_Q \in T$ and $C = \cap_{t \in T} A_1 \circ t$. Note that $T$ is finite and that $C$ is a $\mathbb{Z} Q$-submodule of $A$.

We claim that $A_0/ A_1$ embeds in $Q$. Indeed the group homomorphism $\pi \circ f |_{A_0} : A_0 \to Q$ has kernel $A_1$.
 Since $A_0/ A_1$ embeds in $Q$, we have that $A_0/ A_1$ is finitely generated and  since $A/ A_0$ is finite, we obtain that $A/ A_1$ is finitely generated  as an abelian group. Then for any $q \in Q$ we have an isomorphism of additive abelian groups
 $A/ A_1 \simeq (A \circ q)/ (A_1 \circ q)= A / (A_1 \circ q)$ that sends $a + A_1$ to $ a \circ q + A_1 \circ q$. In particular  $A / (A_1 \circ q)$ is a finitely generated abelian group.
 
 Now, for each $t \in T$, consider the projection map of abelian groups $f_t : A \to A / (A_1 \circ t)$ and {so} for the map $f : A \to \prod_{t \in T}  A / (A_1 \circ t)$ given by $f(a) = \prod_{t \in T} f_t(a)$ we have that $Ker(f) = C$ and   the codomain of $f$ is a finitely generated abelian group.
 Then $A/ C$ is finitely generated and $f(C) \subseteq f(A_1) = 1_G$.
\end{proof}

We denote by $Aug$ the augmentation ideal of the relevant group algebra.

\begin{lemma} \label{L3} Let $G = A \rtimes Q$ be a finitely generated group, where $A$ and $Q$ are abelian. Let $f : H \to G$ be a virtual endomorphism such that $H = A_0 \rtimes Q_0$ and $f(Q_0) \subseteq Q$. Then $f(A_0 \circ Aug( \mathbb{Z} Ker f|_{Q_0})) = 1_G$. 
\end{lemma}

\begin{proof} Suppose that $q \in Q_0$ and $f(q) = 1$. Then for $ a \in A_0$ we have $f( a \circ ( q - 1)) = f([a, q]) = [ f(a), f(q)] = [f(a), 1] = 1$.
\end{proof}

  We recall the definition of a virtual-endomorphism-finite metabelian group that was first stated in the introduction.

{\bf Definition} {\it  Let $G  = A \rtimes Q$, with $A$ and $Q$ abelian, $G$ finitely generated.    Consider  a finite set of virtual endomorphisms   $$f^{(i)} : A_i \rtimes Q_i \to G,$$  such that for $1 \leq i \leq k$ we have 
   
   1) $f^{(i)} (A_i) \subseteq A$,  
   
   2) there is NOT a positive integer $m_i$ such that $m_i A \subseteq A_i$ and $f^{(i)}(m_i A) = 0$,
   
   3)  $f^{(i)}({Q_i}) \subseteq Q$, 
   
   4) $f_0^{(i)} = f^{(i)} |_{Q_i}$ is injective.
   
  We say that $G  = A \rtimes Q$ is virtual-endomorphism-finite if for any finite set of virtual endomorphisms as above    we have that $\{f_0^{(i)} \}_{1 \leq i \leq k}$ generates a {\bf finite} group of injective homomorphisms $\widetilde{Q} \to Q$, where  $\widetilde{Q}$ is a subgroup of finite index in $\cap_{1 \leq i \leq k} \ Q_i$. }

  \begin{theorem} \label{MainThm}  Let $G = A \rtimes Q$ be a group, where $A$ and $Q$ are abelian,  $Q = \mathbb{Z}^s$, $s \geq 2$.
We view $A$ as a right $\mathbb{Z} Q$-module via conjugation and assume that

1) $A$ is a cyclic  $\mathbb{Z} Q$-module, say $A \simeq \mathbb{Z} Q/ I$, $A$ is a $\mathbb{Z}$-torsion-free  integral domain and $Krull dim (A) = 2$;

2)  for every prime number $p$ the ring $A/ pA$ is  an  infinite  integral domain;

3) the image of a non-trivial element of $Q$ in the field of fractions of $A$ is not algebraic over $\mathbb{Q}$. In particular $C_Q(A) = 1_Q$;

4) $G$ is finitely presented;

5) $G$ is virtual-endomorphism-finite.

Then $G$ is not a self-similar group.
\end{theorem}

\begin{proof}

   Note that by Theorem \ref{fin-pres-0} $\Sigma_A^c(Q) = S(Q) \setminus \Sigma_A(Q)$ does not contain antipodal points, hence condition 5) from Theorem \ref{oldA1} holds.

Suppose that  $$f^{(i)} : H_i \to G = A \rtimes Q$$ is a virtual endomorphism (i.e. $f^{(i)}$ is a group homomorphism and $[G : H_i] < \infty$) for  $ 1 \leq i \leq k$ and these virtual endomorphisms show that $G$ is self-similar i.e. if $K$ is a normal subgroup of $G$ such that $f^{(i)}(K) \subseteq K$ for every  $ 1 \leq i \leq k$   then $K = 1_G$. We can define $$A_i = H_i \cap A, Q_i = H_i \cap Q$$ and substitute $H_i$ with the subgroup of finite index $A_i \rtimes Q_i$.

 As before we can assume that $A_i$ is a $\mathbb{Z} Q$-submodule of $A$, otherwise substitute $A_i$ with $\cap_{q \in T_i} A_i \circ q$, where $A_i \circ q $ by definition is $q^{-1} A_i q$ and $T_i$ is a transversal of $Q_i$ in $Q$  such that $1_Q \in T_i$, so $T_i$ is finite.

 Note that, for each  $ 1 \leq i \leq k$, we have one of the following possibilities :

1) Suppose that $f^{(i)}(A_i) \not\subseteq A$. Then  for $$J_i = \cap_{q \in T_i} \{ a \in A_i \ | \ f^{ (i)}(a) \in A \} \circ q$$
 we have
 by the proof of Lemma \ref{L2} that
$$J_i \subseteq Ker( f^{ (i)}) \hbox{ and }A/ J_i \hbox{ is finitely generated as abelian group (via }+ \hbox{).}$$  Then $J_i \not= 0$ and $J_i$ is a $\mathbb{Z} Q$-submodule of $A$, hence an ideal of $A$.

2)  Suppose that $f^{(i)}(A_i) \subseteq A$. The main idea in the proof is to produce a non-trivial normal subgroup $C$ of $G$ such that $C \subseteq A$ and such that  $f^{(j)}(C) \subseteq C$  for all  $ 1 \leq j \leq k$. Since $C \subseteq A$ we can use freely Lemma \ref{L1} and we can assume from now on that $f ^{(i)}(Q_i) \subseteq Q$. 

2.1)Suppose that the restriction of $f^{(i)}$ on  $Q_i$ is not injective and  using Lemma \ref{L3}    we set
$$
J_i = A_i \circ Aug (\mathbb{Z} Ker ( f^{ (i)} |_{Q_i})) \subseteq Ker (f^{ (i)}).
$$ Note that $J_i \not= 0$ is a $\mathbb{Z} Q$-submodule of $A_i$ since $A_i$ is a $\mathbb{Z} Q$-submodule of $A$.

2.2) $f^{ (i)}(A_i) \subseteq A$, $f^{(i)}(Q_i) \subseteq Q$ and the restriction of $f^{(i)}$ on  $Q_i$ is  injective. Then there is an integer $m_i > 0$ such that $m_i A \subseteq A_i$ and  by Theorem \ref{oldA1} and Theorem \ref{rigid}  there are two possibilities:

2.2.1)  $f^{ (i)}(m_i A) = 1_G = 0_A$, then define
$$ 0 \not= J_i = m_i A \subseteq  Ker (f^{ (i)}).
$$
Again $J_i$ is a $\mathbb{Z} Q$-submodule of $A$.

2.2.2) $0 \neq f^{(i)}(m_iA) \subseteq m_i A$  and furthermore there is not a positive integer  $\widetilde{m}_i$ such that $\widetilde{m}_i A \subseteq A_i$ and $f(\widetilde{m}_i A) = 0$, otherwise we are in case 2.2.1). Again,  define
$${J_i = }m_i A  \subseteq A_i.$$ 

Furthermore for every integer $m$ that is divisable by $m_i$ we have $f^{ (i)} (m A) \subseteq m A$. 

 So, suppose $f^{(1)}, \ldots, f^{ (k)}$ are virtual endomorphisms as above. Suppose $f^{(1)}, \ldots, f^{ (c)}$ are of type 1), 2.1) or 2.2.1)  and $f^{(c+1)}, \ldots, f^{ (k)}$ are of type 2.2.2) and neither is of type 1), 2.1), 2.2.1).

Since $G$ is virtual-endomorphism-finite we have that  the set $\{f_0^{(i)} \ | \ c+1 \leq i \leq k\}$  generates  a finite group $O$ of injective maps $\widetilde{Q} \to Q$, where  $\widetilde{Q}$ is a subgroup of finite index in $ \cap_{c+1 \leq i \leq k}   Q_i$.  
 This implies that, extending by linearity the elements of $O$ to $\mathbb{Z} \widetilde{Q} \to \mathbb{Z} Q$ there is a subring $R$ of $\mathbb{Z}\widetilde{Q}$, such that the elements of $R$ are fixed by the elements of the group $O$  and the extension $R \subseteq \mathbb{Z}\widetilde{Q}$ is integral. Since $[Q : \widetilde{Q}] < \infty$ we have that $\mathbb{Z} Q$ is integral over $\mathbb{Z} \widetilde{Q}$, hence $\mathbb{Z} Q$ is integral over $R$. Let $B$ denote the image of $R$ in $A$, then the  extension $B \subseteq A$ is integral. 
 
 \medskip
  {\bf Claim} {\it Let $J\not= 0$ be an ideal of $A$, then $B \cap J \not= 0$.}

 \medskip
 {\bf Proof}.
 Suppose $B \cap J = 0$.
 Then $B = B/ (B \cap J) \subseteq A/ J$ is an integral extension, hence
 $Krulldim(A) = Krulldim (B) = Krulldim(A/ J) < Krulldim(A)$, a contradiction.  This completes the proof of the Claim.

If $c = k$ then   we set $C =  J_1 \ldots J_k$. By construction $C$ is a $\mathbb{Z} Q$-submodule of $A$  i.e. is an ideal of $A$.
 Since   $f^{ (i)}$ is of type 1), 2.1) or 2.2.1) then
 $$f^{(i)}(C) \subseteq f^{(i)}(J_i) =  0_A = 1_G \subseteq C
 \hbox{ for } 1 \leq i \leq k.$$

If $c < k$,
set $J = J_1 \ldots J_c$ and note that $J \not= 0$ since each $J_i \not= 0$ and $A$ is a domain.
Set
$$m = \prod_{c+1 \leq i \leq k} m_i$$

Then we define $C = mb A$, where  $0 \not=b \in J \cap B$.   If $i > c$ and since $b \in B$ we have
$$f^{ (i)} ( C) =f^{ (i)} ( mbA) = f^{ (i)} ( mAb)=  f^{ (i)} ( mA) b \subseteq mAb = mbA = C. 
$$
Note that $f^{ (i)} ( mAb)=  f^{ (i)} ( mA) b$ since $b \in B$.

If $i \leq c$ we have that since $ C  \subseteq bA \subseteq J \subseteq J_i$ 
$$f^{ (i)}(C) \subseteq f^{ (i)} (J_i)  = 0_A = 1_G \subseteq C$$

Thus we have that for every $i$, $f^{ (i)}(C) \subseteq C$. 

Note that by construction  each ${J_i} \not= 0$ and $m A \not= 0$. These together with the fact that $A$ is a domain of zero characteristic, implies that in both cases $c = k$ and $c < k$ we have that  $C$ is a non-trivial $\mathbb{Z} Q$-submodule of $A$, hence is a normal non-trivial subgroup of $G$, a contradiction.

\end{proof}

\section{The proof of the Main Theorem} \label{Proof-main}

In order to apply Theorem \ref{MainThm} we need to verify condition 5) from Theorem \ref{MainThm}.
 Suppose that $\widetilde{I}$ is a finite set and $$f^{(i)} : A_i \rtimes Q_i \to G$$ is a virtual endomorphism such that for every $i \in \widetilde{I}$ we have that
 
  1) $f^{(i)} (A_i) \subseteq A$,  
   
   2) there is not a positive integer $m_i$ such that $m_i A \subseteq A_i$ and $f^{(i)}(m_i A) = 0$,
   
   3)  $f^{(i)}({Q_i}) \subseteq Q$, 
   
   4) $f_0^{(i)} = f^{(i)} |_{Q_i}$ is injective.

 We aim to show that  $f_0^{(i)}$, $i \in \widetilde{I}$ generate a finite group of injective homomorphisms $\widetilde{Q} \to Q$, where  $\widetilde{Q}$ is a subgroup of finite index in $\cap_{i \in \widetilde{I}} Q_i$.

As in the proof of Theorem \ref{oldA1} we have that each $f_0^{(i)}$ induces a permutation of a finite subset $\partial \Sigma_{A_i}^c(Q_i)$ of $\Sigma_{A_i}^c(Q_i)$ (see Step 1 from the proof of Theorem  \ref{oldA1}, where $f_*$ permutes the finite set  $\partial \Sigma_{A_0}^c(Q_0)$)  and   $\Sigma_{A_i}^c(Q_i)$ is in bijection with $\Sigma_A^c(Q)$ induced by the embedding of $A_i$ in $A$ and the embedding of $Q_i$ in $Q$. By the description it is a special finite set of boundary points, so it is the same set for each $i \in  \widetilde{I} $ as it corresponds to the set $E$ of boundary points of $\Sigma_A^c(Q)$.  Let $d$ be the number of elements in this set $E$ of boundary points. Note that by the argument of the proof of Theorem \ref{oldA1} $E$ contains only discrete points of $S(Q)$, actually this is a corollary of the structure of $\Sigma_A^c(Q)$ as a rationally defined spherical  polyhedron that follows from Theorem \ref{thm-valuation1} and Theorem \ref{val-val1}. As well by the proof of Theorem \ref{oldA1} each $(f_0^{(i)})^k = id$, where $ k = d!$ and we need to restrict $f_0^{(i)}$ to a subgroup of finite index in $\cap_{i \in \widetilde{I}} Q_i$ in order $(f_0^{(i)})^{z}$ to be well defined for all $i \in \widetilde{I}$ and $ 0 \leq z \leq k-1$.

Note that each $f_0^{(i)}$ extends to a unique automorphism $g_i$ of $\widehat{Q}$, where $\widehat{Q} = Q \otimes_{\mathbb{Z}} \mathbb{Q}$  in additive notation for $Q$ but actually it is better to use a multiplicative notation for $Q$ as the image of $Q$ in $A$ is considered as a multiplicative  subset of $A$. Note that $\widehat{Q}$ is the Malcev completion of $Q$.

Now we have to show that the group $T$ generated by $\{ g_i \ | \ i \in \widetilde{I} \}$ is finite. Note that $Hom(Q, \mathbb{R})$ is naturally isomorphic to $Hom( \widehat{Q}, \mathbb{R})$. Recall that each $g_i$ permutes the finite set $E$, and this induces a group homomorphism $$\rho : T \to Sym(E) \simeq S_d$$ where $S_d$ is the symmetric group on $d$ elements.

Let $g$ be an element of the kernel of $\rho$. Then using that each $g_i$ has finite order, we can write $g = g_{i_1} \ldots g_{i_j}$ for some $i_1, \ldots, i_j \in \widetilde{I}$. Recall that by step 3 from the proof of Theorem \ref{oldA1} for every subgroup $H^{(i)}$ of finite index in $Q_i$ there is a ring homomorphism
$$\rho^{(i)}_{H^{(i)}} : \mathbb{Z} H^{(i)}/ (I \cap \mathbb{Z} H^{(i)}) \to \mathbb{Z} M^{(i)}/ (I \cap \mathbb{Z} M^{(i)})$$
induced by $f_0^{(i)}$, where $ M^{(i)} = f_0^{(i)} ( H^{(i)})$.

Then there is a subgroup $Q_0$, that depends on $g$, and is of finite index in $Q$ such that the following groups are well-defined
$$Q_0, Q_1 = f_0^{(i_j)}(Q_0), Q_2 =  f_0^{(i_{j-1})}(Q_1), \ldots,  Q_{k+1} = f_0^{(i_{j-k})} (Q_k), \ldots, Q_j = f_0^{(i_1)}(Q_{j-1})$$
and we have the ring epimorphisms
$$\rho^{(i_j)}_{Q_0},  \rho^{(i_{j-1})}_{Q_1},  \rho^{(i_{j-2})}_{Q_2}, \ldots,  \rho^{(i_{1})}_{Q_{j-1}}$$
Then the composition
$$
\mu: = \rho^{(i_{1})}_{Q_{j-1}} \ldots \rho^{(i_{j-1})}_{Q_1}\rho^{(i_j)}_{Q_0} : \mathbb{Z} Q_0 / (I \cap \mathbb{Z} Q_0) \to \mathbb{Z} Q_j / (I \cap \mathbb{Z} Q_j)
$$ is an epimorphism of rings.
Set
 $$f_0 = f_0^{(i_1)} \ldots f_0^{(i_j)} : Q_0 \to Q$$
 and note that the extension of $f_0$ to an automorphism of $\widehat{Q}$ is precisely $g$.
 As in step 1 from the technical Theorem \ref{oldA1} using Theorem   \ref{rigidity1} $Hom(Q, \mathbb{R})$ is spanned as a $\mathbb{R}$-vector space   by characters $\chi_i$  whose projections to $S(Q)$ are inside $\Sigma_A^c(Q)$ but  any such $\chi_i$  is a $\mathbb{R}$-linear combination of characters whose projection to $S(Q)$ is inside $E =\partial \Sigma_A^c(Q)$. Thus there is a basis of $Hom(Q, \mathbb{R})$ as a vector space over $\mathbb{R}$  whose projection to $S(Q)$ is inside $E$.

Since $g \in Ker(\rho)$,  $g$ fixes the elements of $E$. This together with the fact that $Hom(Q, \mathbb{R})$ is spanned by elements whose projection to $S(Q)$ is inside $E$ and that  $g$ is induced by $f_0$ implies that  $Q$ has a multiplicative basis $x_1, \ldots, x_s$ such that  $$f_0(x_i^m) = x_i^{c_i}  \hbox{ for all }1 \leq i \leq s,$$
for some positive integer $m$ such that  $ Q^m \subseteq Q_0$ and  for appropriate integers $\{ c_i \}$.
Then for $B = f_0(Q^m)$ the restriction of $\mu$ to $\mathbb{Z} Q^m / (I \cap \mathbb{Z} Q^m)$ gives a ring epimorphism
$$\theta: \mathbb{Z} Q^m / (I \cap \mathbb{Z} Q^m) \to \mathbb{Z} B / (I \cap \mathbb{Z} B)$$
Note
that $\theta$ is induced by $f_0$ i.e. $\theta |_{\overline{Q^m}} = \overline{f_0 |_{Q^m}}$.
Then by the homothety rigidity property $f_0 = id$. Thus any element of the kernel of $\rho$ is the identity. 

Finally since $T$ is finite, for every element of $T$ we write it as a product $ g_{i_1} \ldots g_{i_j}$ and then we have corresponding $Q_0$. Then intersecting all these $Q_0$'s we get a subgroup $\widetilde{Q}$ of finite index in $Q$ such that $f_0 = f_0^{(i_1)} \ldots f_0^{(i_j)} : \widetilde{Q} \to Q$ is well defined simultaneously for all $f_0$. This completes the proof of the Main Theorem.

\section{Examples}  \label{examples}

1) The group $\mathbb{Z} \wr \mathbb{Z} \simeq A \rtimes Q $, where $$ A = \mathbb{Z}[x, 1/x],Q = \langle q \rangle \simeq \mathbb{Z}$$ with $q$ acting (via conjugation)  on $A$ by multiplication with $x$, is not finitely presented, is not transitive self-similar \cite{Brazil2} but is intransitive  self-similar \cite{Brazil}. Note that $A$ has Krull dimension 2 but is not homothety rigid and $s = 1$.  The group $G = A \rtimes Q$ satisfies conditions 1,2,3 from our Main Theorem but does not satisfy condition 4.

2) The group  $\mathbb{Z} \wr \mathbb{Z}$ embeds in $G = A \rtimes Q,$ where $$A = \mathbb{Z}[x^{\pm 1}, 1/(x+1)], Q = \langle q_1, q_2 \rangle \simeq \mathbb{Z}^2$$ and $Q$ acts on $A$ via conjugation with $q_1$ acting by multiplication with $x$ and $q_2$ acting by multiplication with $x+1$.

By \cite{B-S} $G$ is finitely presented. Indeed in the example of Section \ref{Sigma} $\Sigma_{A}^c(Q) = \partial\Sigma_A^c(Q)$ consists of three isolated points and by Theorem \ref{fin-pres-0} $G$ is finitely presented.

Moreover, $A$ has Krull dimension 2 and no element of $Q \setminus \{ 1 \} \subseteq K$ is algebraic over $\mathbb{Q}$, where $K$ is the field of fractions of $A$, so by Theorem \ref{rigid}, $A$ is $Q$-homothety rigid. Furthermore for every prime $p$ we have that $A / pA$ is an integral domain of Krull dimension 1, in particular is infinite.  Thus $G$ satisfies all four condition from our Main Theorem and so is not self-similar.

\end{document}